\documentclass{amsart}

\usepackage[utf8]{inputenc}

\usepackage{lmodern}
\usepackage{color}
\usepackage{amsmath}
\usepackage{amssymb}
\usepackage{amsthm}
\usepackage{bbm}   % For \mathbbm{1}
\usepackage{mathtools}
\usepackage{orcidlink}
\usepackage{colonequals}
\usepackage{bm}
\usepackage[linesnumbered]{algorithm2e}
\RestyleAlgo{ruled}

\usepackage[maxbibnames=99,  % Never write 'et al.' in a bibliography
           giveninits=true, % Abbreviate first names
           sortcites=true,  % Sort citations when there are more than one
           backend=biber]{biblatex}
\usepackage{microtype}
\usepackage{todonotes}

\usepackage{array}
\usepackage{tikz}
\usetikzlibrary{spy}
\usetikzlibrary{external}
\usepackage{graphicx}
\usepackage{subfigure}

\usepackage{pgfplots}
\usepgfplotslibrary{fillbetween,colormaps}
\pgfplotsset{compat=newest}

\usepackage{enumitem}
\usepackage{hyperref}

\usepackage[capitalize,nameinlink,noabbrev]{cleveref}

\allowdisplaybreaks[0]

\newcommand{\R}{\mathbb R}
\newcommand{\N}{\mathbb N}

\DeclarePairedDelimiter{\abs}{\lvert}{\rvert}
\DeclarePairedDelimiter{\norm}{\lVert}{\rVert}
\DeclareMathOperator*{\argmin}{arg\,min}

\newcommand{\cone}{\operatorname{cone}}
\newcommand{\diag}{\operatorname{diag}}

\newcommand{\product}{\operatorname{prod}}
\newcommand{\skewsig}{\operatorname{\mathcal{S}_{\Sigma,\lambda}}}

\newcommand{\dist}[1]{\operatorname{d}_{{#1}}}
\newcommand{\distlog}[1]{\operatorname{d}_{#1}^\textup{log}}
\newcommand{\disthyp}[1]{\operatorname{d}_{#1}^\textup{hyp}}

\newcommand{\SL}{\textup{SL}}
\renewcommand{\sl}{\textup{sl}}

\newtheorem{theorem}{Theorem}[section]
\newtheorem{definition}[theorem]{Definition}
\newtheorem{lemma}[theorem]{Lemma}
\newtheorem{remark}[theorem]{Remark}
\newtheorem{conjecture}[theorem]{Conjecture}

\theoremstyle{definition} % do not print examples italic

\graphicspath{{gfx/}}

\hypersetup{pdfauthor={Patrick Jaap, Oliver Sander},
            pdftitle={How to project onto SL(n)}
            }

\title{How to project onto $\SL(n)$}

\author[Jaap]{Patrick Jaap}
\address{Patrick Jaap\\
Weierstraß-Institut\\
Mohrenstraße 39\\
10117 Berlin\\
Germany \\
\orcidlink{0000-0001-8567-9988}~\href{https://orcid.org/0000-0001-8567-9988}{0000-0001-8567-9988}}
\email{patrick.jaap@wias-berlin.de}
\thanks{This work was funded by the Deutsche Forschungsgemeinschaft
(DFG, funder DOI: \url{https://dx.doi.org/10.13039/501100001659})
as part of the Schwer\-punktprogramm SPP\,1962: ``Non-smooth and Complementarity-based Distributed Parameter Systems -- Simulation
and Hierarchical Optimization''; Project~16: ``Nonsmooth Multi-Level Optimization Algorithms for Energetic Formulations of Finite-Strain Elastoplasticity''
}

\author[Sander]{Oliver Sander}
\address{Oliver Sander\\
Technische Universität Dresden\\
Institut für Numerische Mathematik\\
Zellescher Weg 12--14\\
01069 Dresden\\
Germany \\
\orcidlink{0000-0003-1093-6374}~\href{https://orcid.org/0000-0003-1093-6374}{0000-0003-1093-6374}}
\email{oliver.sander@tu-dresden.de}

\begin{document}

\begin{abstract}
We consider the closest-point projection with respect to the Frobenius norm
of a general real square matrix to the set $\SL(n)$ of matrices with unit determinant.
As it turns out, it is sufficient to consider diagonal matrices only.
We investigate the structure of the problem both in Euclidean coordinates
and in an $n$-dimensional generalization of the classical hyperbolic coordinates
of the positive quadrant.
Using symmetry arguments we show that the global minimizer is contained
in a particular cone. Based on different views of the problem, we propose
four different iterative algorithms, and we give convergence results for all of them.
Numerical tests show that computing the projection costs essentially as much
as a singular value decomposition.
Finally, we give an explicit formula for the first derivative of the projection.
\end{abstract}


\maketitle

\section{Introduction}

Let $n \ge 2$ be a natural number, and define the special linear group
\begin{equation*}
	\SL(n)\colonequals \big\{ P\in\R^{n\times n}\colon~\det(P)=1\big\}.
\end{equation*}
For a given matrix $A\in\R^{n\times n}$ we consider the problem of finding
a $P\in\SL(n)$ that minimizes the squared Frobenius distance%
\footnote{The factor $\frac{1}{2}$ simplifies various expressions appearing
later in the text. In an abuse of notation we will call $\dist{A}$ the
squared Frobenius distance despite this factor.
}
to $A$
\begin{equation}
\label{eq:squared_frobenius_distance}
 \dist{A}(P)
 \colonequals
 \frac{1}{2} \norm{A-P}^2_F
	\colonequals
 \frac{1}{2} \sum_{i,j=1}^n (A_{ij}-P_{ij})^2.
\end{equation}
Since $\dist{A}$ is continuous, coercive, and bounded from below,
it does have a global minimizer on the closed set $\SL(n)$.

The only previous work on projections onto $\SL(n)$ that we are aware of
is an article by \textcite{baaijens_draisma:2015}.  Using techniques from
algebraic geometry, they prove that for any $A \in \R^{n \times n}$ with pairwise
distinct singular values,
the distance $\dist{A}$ restricted to $\SL(n)$ has $n2^{n-1}$ stationary points
(some of them possibly complex).  \citeauthor{baaijens_draisma:2015} also give an algorithmic
procedure to obtain such a stationary point, and conjecture a way to find the
global minimizer of the distance.
However, this procedure is fairly complicated, and it involves finding roots
of a high-degree polynomial, which is difficult to do
in practice.

We look at the same projection problem from the point of view of numerical analysis.
Finding one of the local minimizers among the $n2^{n-1}$ stationary points of $\dist{A}$
approximately is reasonably straightforward, using algorithms from the mature field of
equality-constrained optimization~\cite{nocedal_wright:2006}.
However, the problem has a considerable amount of structure, which
can be used to construct algorithms that are more efficient and robust
than off-the-shelf ones.  It is not a geometric program in the sense of~\cite[Chapter~4.5]{boyd_vandenberghe:2004},
because the objective function is not a posynomial.
The problem contains various symmetries, which provide some hints
about the location of the global minimizer, but reliably finding such a global minimizer seems
out of reach (except for the case $n=2$).

We begin by revisiting the fact that it is sufficient to consider the space
of diagonal matrices only.
As already noted by \textcite[Remark~4.2]{baaijens_draisma:2015}, the problem
is equivalent to minimizing the
Euclidean distance in $\R^n$ of a point $a \in \R^n$ to the set%
\footnote{
Not to be confused with the Lie algebra $\mathfrak{sl}(n)$ of $\SL(n)$, which does not
appear in this manuscript.
}
\begin{equation*}
 \sl(n)
 \colonequals
 \Big\{ x \in \R^n \; : \; \product(x) = 1 \Big\},
 \qquad
 \product(x) \colonequals \prod_{i=1}^n x_i.
\end{equation*}
This is a considerable simplification.  In particular, the connected
admissible set $\SL(n)$ is replaced by one with many connected components,
which are all identical modulo simple symmetry transformations.
Considering only one of them, the set $\sl(n) \cap \R^n_{\ge 0}$ of admissible points
with nonnegative coefficients, is therefore sufficient.

The set $\sl(n) \cap \R^n_{\ge 0}$ turns out to be the boundary of a strictly convex
subset of $\R^n$. The problem of projecting a point $a \in \R^n$ onto $\sl(n) \cap \R^n_{\ge 0}$
therefore changes its character
depending on whether $a$ is inside or outside of this set, i.e., whether $\product(a) \ge 1$
or not.  Using simple symmetry considerations we then show that the
global minimizer is in the cone of points that have the same coordinate ordering as $a$.
We prove that the global minimizer is the only minimizer in that cone
if $n=2$, and show that this is not true in higher dimensions.
The arguments of the counterexample lead to a first algorithm that finds
the projection by a bisection search for the Lagrange multiplier
of the constraint to $\sl(n)$.

In an attempt to arrive at simpler iterative solvers, we then move to
logarithmic coordinates to replace the non-Euclidean admissible set $\sl(n) \cap \R_{\ge 0}^n$
by a vector-space (at the price of complicating the objective functional).
For such a formulation, a constrained Newton method is a natural choice.
A further coordinate transformation allows to eliminate the constraint completely.
The resulting coordinates turn out to be a generalization of the hyperbolic coordinates
of the positive quadrant of $\R^2$ formed by the hyperbolic angle and the geometric mean.

In total, we present four different
numerical algorithms for the computation of a minimizer of the squared Frobenius distance.
These algorithms are all
well-defined even for matrices $A$ with zero or duplicate singular values.
For two of them we can show that the sequences of iterates they produce
stay in the cone that contains the global
minimizer. We perform a sequence of tests of these algorithms,
using sets of random dense matrices of different sizes.
Unfortunately, each algorithm requires a singular value decomposition to go
from $\R^{n \times n}$ to $\R^n$, and this will turn out to dominate the run time.

One has to note that our choice of algorithms is not the end of the story.
Projecting onto $\SL(n)$ is also possible with algorithms from the field
of optimization on manifolds~\cite{absil_mahony_sepulchre:2009},
difference-of-convex programming~\cite{deoliveira:2020},
or using nonconforming algorithms similar to \cite{gao_vary_ablin_absil:2022}.
The readership is invited to investigate these approaches in detail.

Our interest in projecting onto $\SL(n)$ originates from theories of finite-strain
elastoplasticity~\cite{Mielke2002}, where the plastic strain, a field of
$3 \times 3$ matrices, is frequently constrained to $\SL(3)$ to model
plastic incompressibility.  One way to approximate such fields are projection-based
finite elements \cite{grohs_hardering_sander_sprecher:2019}, which involve
a projection onto $\SL(n)$.
As the projection needs to happen once for each quadrature point
of the finite element grid, the efficiency of the projection algorithm
is crucial. In addition, projection-based finite element methods require the derivative
of the projection map $A \mapsto P$, which can be computed easily using
the implicit function theorem.
For the convenience of the reader we spell out the formula for the derivative in the last chapter.
As a by-product we learn under what circumstances this derivative is well defined.

\begin{remark}[Radial scaling]
\label{rem:scaling}
For invertible matrices, it is tempting to project onto $\SL(n)$ by simply dividing $A$
by the $n$-th root of its determinant
\begin{equation*}
 A \mapsto P_\textup{scale}(A) \colonequals A (\det A)^{-\frac{1}{n}} \in \SL(n).
\end{equation*}
However, the element in $\SL(n)$ obtained in this way can be far from optimal.
To see this, consider the case $n=2$ and define the sequence $(A_k)_{k\in\N}$ in $\R^{2\times 2}$ by
\begin{equation*}
 A_k
 \colonequals
 \begin{pmatrix} k & 0 \\ 0 & \frac2k \end{pmatrix}.
\end{equation*}
Then we have $\det A_k = 2$, and hence
$\norm{A_k - P_\textup{scale}(A_k)}_F^2 = (1-\frac{1}{\sqrt2})^2 \norm{A_k}^2_F \stackrel{k\to\infty}{\longrightarrow} \infty$.
The optimal projection, on the other hand, must be at least as good as
\begin{equation*}
 P_k
 \colonequals
 \begin{pmatrix} k & 0 \\ 0 & \frac1k \end{pmatrix}
 \in \SL(n),
\end{equation*}
and for this we have
\begin{equation*}
 \norm{A_k - P_k}^2_F
 =
 \frac{1}{k^2} \stackrel{k\to\infty}{\longrightarrow} 0.
\end{equation*}
\end{remark}

\tableofcontents

\section{Transformation to diagonal matrices}
\label{sec:diagonalization}

The problem is simplified considerably by realizing that it is sufficient
to consider the space of diagonal matrices only.
To this end we rewrite the minimization problem in $\SL(n)$ as a
minimization problem in $\R^{n\times n}$ with the constraint $\det P =1$.
For a given $A \in \R^{n \times n}$ we define the Lagrange functional
\begin{equation*}
	\mathcal{L} :\R^{n\times n}\times \R \to \R,
	\qquad
	\mathcal{L}(P,\lambda) \colonequals \frac12 \norm{P-A}_F^2 + \lambda(\det P -1).
\end{equation*}
All local minimizers in $\SL(n)$ of the squared distance functional $\dist{A}$
defined in~\eqref{eq:squared_frobenius_distance}
are stationary points of $\mathcal{L}$ when paired with the corresponding Lagrange
multiplier $\lambda$.  Suppose that $P \in \R^{n \times n}$ is such a local
minimizer.  Then
\begin{equation}\label{eq:lagrange}
	\nabla \mathcal{L}(P,\lambda) = \begin{pmatrix}
		P-A + \lambda \nabla\det P \\ \det P - 1
	\end{pmatrix} = 0.
\end{equation}
From the second line it follows that $P$ is invertible and
Jacobi's formula~\cite{magnus_neudecker:1988} then gives
\begin{equation*}
	\nabla\det P = (\det P) P^{-T} = P^{-T}.
\end{equation*}
Using this in the first line of~\eqref{eq:lagrange} implies that
\begin{equation}
\label{eq:relation_A_to_P}
	A = P + \lambda P^{-T}.
\end{equation}

Let now $U,V \in \R^{n \times n}$ be two orthogonal matrices that diagonalize
$P$ in the sense that $P = U \Sigma_P V^T$ for some diagonal matrix $\Sigma_P$.
Such a pair of matrices always exists in form of the singular value decomposition.
Using the diagonalization in~\eqref{eq:relation_A_to_P} we then get that
\begin{equation} \label{eq:A-and-P-diagonalized}
 A
 =
 U\big(  \Sigma_P +\lambda \Sigma_P^{-1} \big)V^T
 \equalscolon
 U \Sigma_A V^T,
\end{equation}
which means that $U$ and $V$ also diagonalize $A$.
Conversely, if two orthogonal matrices $U$ and $V$ diagonalize $A$ such that
$A = U \Sigma_A V^T$ for a diagonal matrix $\Sigma_A$, and if the diagonal
matrix $\Sigma_P$ solves $\Sigma_P + \lambda \Sigma_P^{-1} = \Sigma_A$,
then $P \colonequals U \Sigma_P V^T$ solves~\eqref{eq:relation_A_to_P}.
Together with the equality constraint we obtain the Lagrange conditions
for the diagonal matrices
\begin{align}
\label{eq:lagrange_diagonal}
 \begin{pmatrix}
  \Sigma_P -\Sigma_A + \lambda \Sigma_P^{-1} \\
  \det \Sigma_P - 1
 \end{pmatrix}
 =
 0.
\end{align}

Comparing~\eqref{eq:lagrange_diagonal} with~\eqref{eq:lagrange} we see that the sets of solutions
are isomorphic.  What is more, since the Frobenius norm
is invariant under orthogonal transformations from the left and from the right
we have
\begin{equation*}
 \norm{P-A}_F
 =
 \norm{\Sigma_P - \Sigma_A}_F.
\end{equation*}
Therefore, if $P$ is a (global) minimizer of $\dist{A}$ in $\SL(n)$,
then $\Sigma_P$ is a (global) minimizer of the distance to $\Sigma_A$
in the space of diagonal matrices with unit determinant, and vice versa.
A related result is stated in~\cite[Remark~4.2]{baaijens_draisma:2015}.
To compute $P$ from $A$ it is therefore sufficient to compute $\Sigma_A$
and then to look for a minimizer of $\dist{\Sigma_A}$
in the set of diagonal matrices with unit determinants.

\section{The structure of the admissible set, and the global minimizer}
\label{sec:structure_of_admissible_set}

We have just shown that it is sufficient to consider the projection problem
in the space of diagonal $n \times n$ matrices. To simplify the notation
we now identify this space with $\R^n$ by the canonical operator $\diag :\R^n \to \R^{n \times n}$,
whose inverse $\diag^{-1}$ is well-defined on the set of diagonal matrices.
The operator $\diag$ is an isometry when $\R^n$ is equipped with the Euclidean norm $\norm{\cdot}_2$.
We will denote points in $\R^n$ by lower-case latin letters,
and write $p^{-1}$ for the component-wise inverse of $p \in \R^n$ provided
that it is well defined.

\subsection{The minimization problem in $\R^n$}

Denote the subset of $\R^n$ that corresponds to the diagonal matrices
of unit determinant by
\begin{equation}\label{eq:G}
 \sl(n)\colonequals \Big\{  p\in\R^n:\  \product (p) = 1 \Big\},
 \qquad \text{with} \quad
 \product (p) \colonequals \prod_{i=1}^{n} p_i.
\end{equation}
From now on we consider the problem of finding a minimizer of the functional
\begin{equation*}
 \dist{a} : \R^n \to \R,
 \qquad
 \dist{a}(p) \coloneqq \frac{1}{2} \norm{p-a}_2^2,
 \qquad
 p \in \sl(n).
\end{equation*}
The stationarity conditions for this are
\begin{align}
\label{eq:lagrange_vector_formulation}
 \begin{pmatrix}
  p - a + \lambda p^{-1} \\
  \product(p) - 1
 \end{pmatrix}
 =
 0,
\end{align}
which are simply \eqref{eq:lagrange_diagonal} with a different notation.

While being posed on a much smaller space, Problem~\eqref{eq:lagrange_vector_formulation}
yields a set of stationary points that is isomorphic to the set of stationary points
of the original problem. However, the structure of the admissible set has changed.
% The number of connected components is discussed in
% https://math.stackexchange.com/questions/1584740/how-many-binary-words-of-length-n-that-consist-an-even-number-of-zeros
Unlike $\SL(n)$, which is connected~\cite[Exercise~M.8.\,(a)]{artin:2011},
the set $\sl(n)$ consists of $2^{n-1}$ separate components, each characterized by
a fixed subset of even size of variables with negative sign.
Each component is a closed set in $\R^n$, and therefore contains a
minimizer of the distance functional.
This is one source of nonuniqueness of projections onto $\SL(n)$ and $\sl(n)$.

Geometrically, $\sl(n)$ is a differentiable submanifold of $\R^n$.
With componentwise multiplication, it is a Lie group
with identity element $\mathbbm{1} \coloneqq (1,\dots,1)^T$.
Since $\sl(n)$ has codimension~1 in $\R^n$, and since
$\nabla \product(p) = p^{-1} \colonequals (p_1^{-1},\dots,p_n^{-1})^T$
is a normal vector
to $\sl(n)$ at $p \in \sl(n)$, the tangent space $T_p \sl(n)$ at a point $p$
consists of all vectors that are orthogonal to $p^{-1}$, i.e.,
\begin{equation*}
 T_p\sl(n)
 =
 \big\{ v \in \R^n \; : \; \langle p^{-1}, v \rangle = 0 \big\}.
\end{equation*}
In particular, tangent vectors $v$ at $\mathbbm{1}$ are characterized by
$\sum_{i=1}^n v_i = 0$.

From its Lie group structure, the manifold $\sl(n)$ inherits a natural
exponential map
\begin{equation}
\label{eq:exponential_map}
 \exp : T_\mathbbm{1} \sl(n) \to \sl(n),
\end{equation}
defined by application of the usual scalar exponential function componentwise.
It is a diffeomorphism between
$T_\mathbbm{1} \sl(n)$ and the connected component of $\sl(n)$
that contains the identity~$\mathbbm{1}$.
Note that it is not the Riemannian exponential map induced by the surrounding Euclidean space.

\subsection{The positive orthant}

\begin{figure}
\label{fig:sl2_sl3}
 \begin{center}
   \begin{tikzpicture}

 \begin{axis}[
         axis line style={<->,color=black, thick}, % arrows on the axis
         height=0.27\textheight,
         xlabel={$x_1$},
         ylabel={$x_2$},
         xmin=-5,xmax=5,
         ymin=-5,ymax=5,
         xtick={-4,-2,...,4},
         ytick={-4,-2,...,4},
         ]

      \draw[dashed] (-5,0) -- (5,0);
      \draw[dashed] (0,-5) -- (0,5);

      \fill[gray!30] plot[domain=0.2:5,smooth,samples=100] (\x,{1/\x}) -- (5,5)  -- (0.2,5) -- cycle;

      \draw[red,very thick, smooth,samples=100,domain=-5:-0.2] plot(\x,{1/\x});

      \draw[blue,very thick, smooth,samples=100,domain=0.2:5] plot(\x,{1/\x});
 \end{axis}

\end{tikzpicture}
  \qquad
  \includegraphics[height=0.25\textheight]{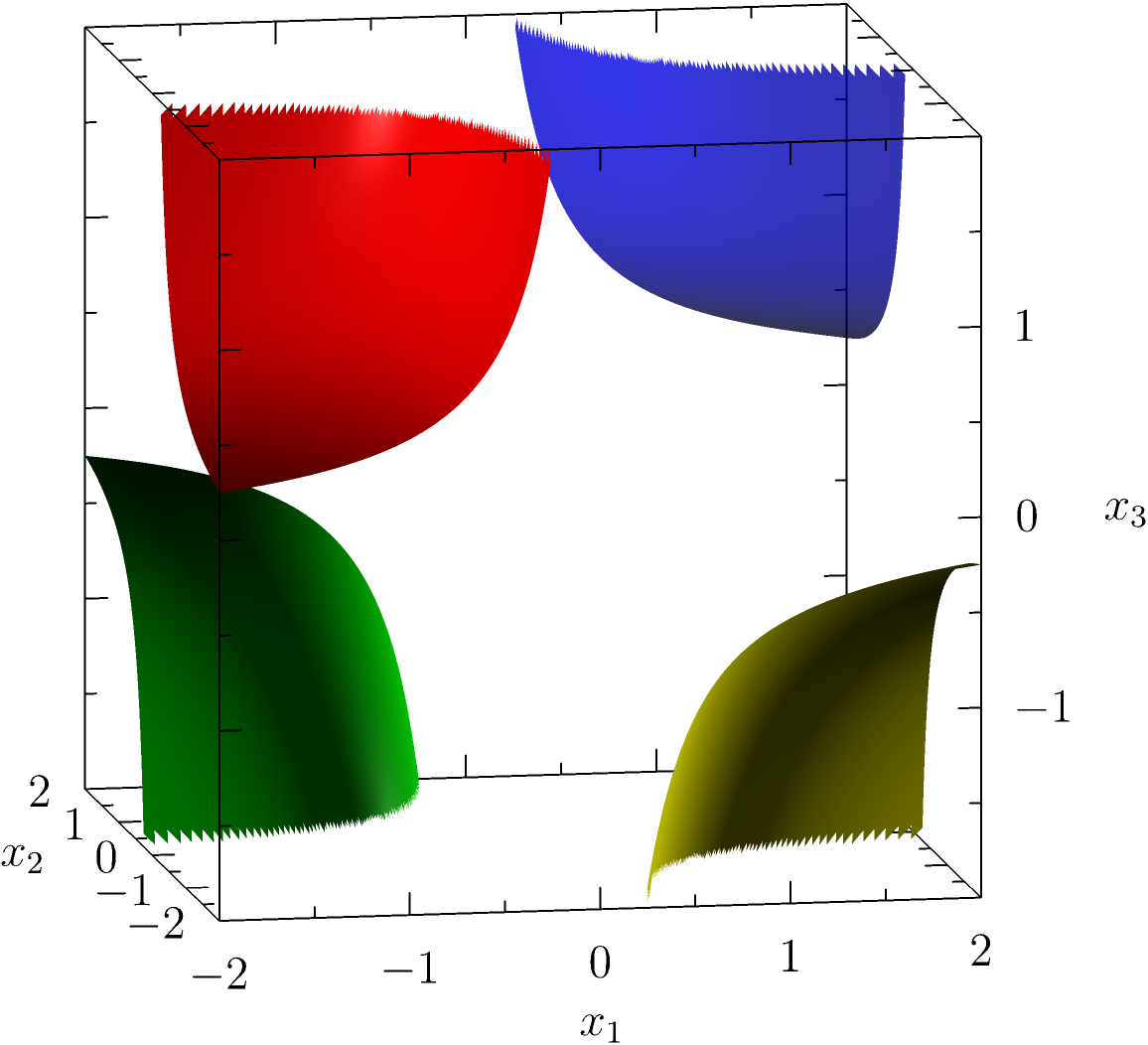}
 \end{center}

 \caption{The sets $\sl(2)$ (left) and $\sl(3)$ (right). The colored lines (left)
  and surfaces (right) show the actual sets $\sl(n)$. The grey area on the left
  is the set $\sl^+(n) \cap \R^n_{\ge 0}$ whose strict convexity is proved
  in Lemma~\ref{lemma:G+-is-convex}. In the right picture, it is the set
  beyond the blue surface.}
 \label{fig:sl2_and_sl3}
\end{figure}
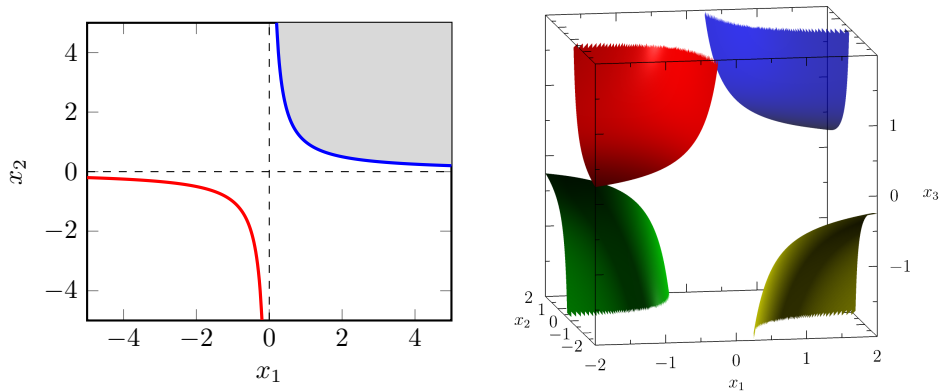

It is straightforward to single out the connected component of $\sl(n)$ that contains
the global minimizer.
Define the positive orthant $\R^n_{\ge 0} \colonequals \{ x \in \R^n \; : \; x_i \ge 0 \; \forall i=1,\dots,n\}$.
Given any matrix $A \in \R^{n \times n}$ with diagonalization $A = U \Sigma_A V^T$,
by adjusting the signs of matrix rows of $U$ and $V$
we can obtain a diagonalization $a = \Sigma_A$ of $A$ that has only nonnegative
diagonal entries---the singular value decomposition.
Since the Euclidean distance $\norm{\cdot}_2$ is invariant under sign flips
of individual components of its argument, we can deduce that
the global distance minimizer must be contained
in the connected component $\sl(n) \cap \R^n_{\ge 0}$ of $\sl(n)$ that has only positive entries.

\begin{lemma}
 Let $a \in \R^n_{\ge 0}$. Then there is a global minimizer $p \in \sl(n)$
 with $p_i > 0$ for all $i\in\{1,\hdots,n\}$.
\end{lemma}
\begin{proof}
 Suppose that for each global minimizer $p$ there is an index $i$
 (depending on~$p$) such that $p_i < 0$.  Let $p$ be such a global minimizer.
 Since $\product(p) = 1 > 0$, there must be at least one other index $j \neq i$
 such that $p_j < 0$ as well.

 Then $(-p_i - a_i)^2 < (p_i - a_i)^2$ and $(-p_j - a_j)^2 < (p_j - a_j)^2$.
 Consequently, if $p^\pm \in \R^n$ is the vector $p$ with the signs of the
 $i$-th and $j$-th components reversed, then $p^\pm \in \sl(n)$ and
 $\norm{p^\pm -a}^2_2 < \norm{p-a}^2_2$. This contradicts the assumption that $p$
 was a global minimizer.
\end{proof}

Figure~\ref{fig:sl2_and_sl3} shows the set $\sl(n)$ for $n=2$ and $n=3$.
One sees that the connected components are not convex.  However, for a point $a$
with $\product a < 1$, the closest point in $\sl(n) \cap \R^n_{\ge 0}$ is also
the closest point in the set $\sl^+(n) \cap \R^n_{\ge 0}$ (with
$\sl^+(n) \colonequals \{x \in \R^n \; : \; \product (x) \ge 1\}$),
and that set is strictly convex.

\begin{lemma}\label{lemma:G+-is-convex}
	For any $n\in\N$ the set $\sl^+(n) \cap \R^n_{\ge 0}$ is strictly convex.
\end{lemma}

\begin{proof}
 The logarithm of the product function
 \begin{equation*}
  \log(\product(x)) = \sum_{i=1}^n \log x_i
 \end{equation*}
 is strictly concave, because the Hesse matrix of $\log(\product(x))$
 is $\operatorname{diag}\big( -\frac{1}{x_1^2}, \dots , -\frac{1}{x_n^2} \big)$,
 and that is negative definite wherever defined.

 Let now $x\neq y\in \sl^+(n) \cap \R^n_{\ge 0}$ and $t \in (0,1)$.
 To show that $t x + (1-t) y$ is  in the interior of $\sl^+(n) \cap \R^n_{\ge 0}$
 we have to prove that $\product(tx + (1-t) y) > 1$. But this follows because
 \begin{align*}
  \log \big( \product(tx + (1-t) y)\big)
  & >
  t \underbrace{\log \product(x)}_{\ge 0} + (1-t) \underbrace{\log \product(y)}_{\ge 0}
  \ge
  0.
 \end{align*}
 Strict monotonicity of the logarithm then implies that $\product\big(tx + (1-t) y\big) > 1$.
\end{proof}

While $\sl^+(n) \cap \R^n_{\ge 0}$ is convex, its complement
$\R^n_{\ge 0} \setminus \sl^+(n)$ is not.
As a consequence, the nature of the minimization problem changes considerably depending on
whether $\product(a)$ is larger or smaller than $1$.  Indeed, if $\product(a) < 1$
then minimizing $\dist{a}$ on $\sl(n) \cap \R^n_{\ge 0}$ is a strictly convex
optimization problem, which has a unique minimizer. If $\product(a) > 1$, however,
the admissible set $\sl(n) \cap \R^n_{\ge 0}$ is not the boundary of a convex set, which makes the
optimization problem harder, and is a second source of non-unique minimizers.
As an example, consider the case $n=2$,
depicted on the left side of Figure~\ref{fig:sl2_and_sl3}.  If $a$ is placed on the positive
diagonal and such that $a_1 = a_2 > 2$, then the problem has two global
minimizers.  This is not perfectly obvious from the picture alone, but will be shown
in Chapter~\ref{sec:uniqueness_in_the_order_cone} below.

As it turns out, the two cases can also be distinguished from the sign
of the Lagrange multiplier.

\begin{lemma}
\label{lem:sign_of_lambda}
 The Lagrange multiplier $\lambda$ is positive in a stationary point if and only if $\product(a) > 1$.
\end{lemma}

\begin{proof}
   Let $a \in \R^n_{\ge 0}$, and let $p \in \sl(n) \cap \R^n_{\ge 0}$ be a stationary point of $\dist{a}$.
  Then by~\eqref{eq:lagrange_vector_formulation}
\begin{equation*}
  a_i = \big( p + \lambda \nabla \product(p) \big)_i = p_i + \lambda p_i^{-1}
  \qquad
  i=1,\dots,n.
\end{equation*}
  If $\lambda \le 0$ all components of $a$ fulfill $0 \le a_i \le p_i$, and therefore, $\product(a) \le \product(p) = 1$.
  On the other hand, if $\lambda > 0$ we have $a_i > p_i > 0$ for all components and $\product(a) > \product(p) = 1$.
\end{proof}

\subsection{The order cone}
\label{sec:global_minimizer}

\begin{figure}
 \begin{center}
  \begin{tikzpicture}
 % Background layer for the order cone
 \pgfdeclarelayer{background}
 \pgfsetlayers{background,main}

 \begin{axis}[
         axis x line=middle,    % put the x axis in the middle
         axis y line=middle,    % put the y axis in the middle
         axis line style={<->,color=black, thick}, % arrows on the axis
         height=0.27\textheight,
         xlabel={$x_1$},
         ylabel={$x_2$},
         xmin=-2.5,xmax=2.5,
         ymin=-2.5,ymax=2.5,
         xtick={-2,-1,...,2},
         ytick={-2,-1,...,2},
         ]

         % Draw the order cone
         \begin{pgfonlayer}{background}
          \fill [color=black!10!white] (-2.5,-2.5) -- (2.5,-2.5) -- (2.5,2.5) -- cycle;
         \end{pgfonlayer}

 \end{axis}
  \end{tikzpicture}
  \qquad
\begin{tikzpicture}
  \begin{axis}[
    height=0.25\textheight,
    xlabel={$x_1$},
    ylabel={$x_2$},
    zlabel={$x_3$},
    xmin=-3, xmax=3,
    ymin=-3, ymax=3,
    zmin=-3, zmax=3,
    view={30}{30}, % gerne an der Perspektive schrauben
    ]

%    \addplot3[
%    surf,
%    samples=50,
%    domain=-3:3,
%    domain y=0:6,
%    color=red!50!black,
%    faceted color =red!50!black,
%    ] (x+y, x+y, x);

    \addplot3[
    surf,
    shader=interp,
    samples=50,
    domain=-3:3,
    domain y=0:6,
    colormap={example}{color=(black!35) color=(black!35)},
    %opacity = 0.5
    ] (x, x-y, x-y);

    \addplot3[
    surf,
    shader=interp,
    samples=50,
    domain=-3:3,
    domain y=0:6,
    colormap={example}{color=(black!40) color=(black!40)},
    %opacity = 0.5
    ] (3, x, x-y);

    % zeiche die übermalten Linien noch einmal neu
    \addplot3[
    samples=2,
    domain=-3:3
    ] (3, x, x);

    \addplot3[
    samples=2,
    domain=-3:3
    ] (3, x, -3);

    \addplot3[
    samples=2,
    domain=-3:3
    ] (3, 3, x);

  \end{axis}

\end{tikzpicture}

 \end{center}

 \caption{The order cones for $n=2$ and $n=3$}
\end{figure}
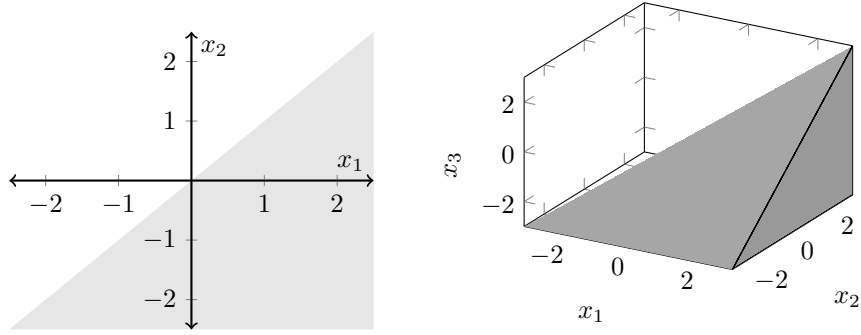

Attempting to locate a global minimizer more precisely, we note that both the distance function
and $\sl(n)$ are symmetric, i.e., invariant
under permutations of the coordinate axes. This allows to assume that $a \in \R_{\ge 0}^n$
has ordered components
 \begin{equation*}
  a_1 \ge a_2 \ge \dots \ge a_n.
 \end{equation*}
We define the set of all points that fulfill this condition.

\begin{definition}
 The order cone is
 \begin{equation*}
  C_n
  \colonequals
  \Big\{ x \in \R^n \; : \; x_1 \ge x_2 \ge \dots \ge x_n \Big\}.
 \end{equation*}
\end{definition}

If $a$ is in this cone, then the cone will also contain a global minimizer.

\begin{lemma} \label{lemma:p-in-cone}
 Let $a\in C_n \cap \R^n_{\ge 0}$. Then there is a  $p \in C_n \cap \sl(n) \cap \R^n_{\ge 0}$
 with $\dist{a}(p) \le \dist{a}(q)$ for all $q\in \sl(n)$.
\end{lemma}
\begin{proof}
 Let $p\in \sl(n)$ be a global minimizer of $\dist{a}$. By the previous section we can assume that $p \in \R^n_{\ge 0}$.
 Suppose that there is an index $i$ with $p_i < p_{i+1}$. In that case, there is an $\varepsilon > 0$
 such that $p_{i+1} = p_{i} + \varepsilon$.
 Moreover, there is a $\delta \ge 0$ such that $a_{i+1} = a_{i} - \delta$.
 Define $\hat{p}$ exactly like $p$, only with swapped components
 $\hat{p}_i = p_{i+1}$ and $\hat{p}_{i+1} = p_i$, and note that
 $\hat{p} \in \sl(n)$.
	Then we have
	\begin{align*}
	\dist{a}(p)^2 - \dist{a}(\hat{p})^2
  &=
  \frac{1}{2} \Big[(a_i-p_i)^2 + (a_{i+1} - p_{i+1})^2 - (a_i - p_{i+1})^2 - (a_{i+1} - p_i )^2 \Big]
  \\
  &=
  \frac{1}{2} \Big[(a_i-p_i)^2 + (a_{i} - \delta - p_{i} - \varepsilon)^2 - (a_i - p_{i} - \varepsilon)^2 - (a_{i} - \delta - p_i )^2 \Big]
  \\
  &=
  \varepsilon\delta \ge 0 .
	\end{align*}
 Hence, swapping the components $p_i$ and $p_{i+1}$ does not increase $\dist{a}$.
 Since $p$ is a global minimizer, we conclude that $\hat p$ is also a global minimizer.
 Following the idea of the bubblesort sorting algorithm, one can swap components
 until arriving at a global minimizer in $C_n \cap \sl(n) \cap \R^n_{\ge 0}$.
\end{proof}

The previous argument established that to find a global minimizer we can transform $a$
to be in the positive order cone by permutations of components.
A global minimizer can then be found in the same cone, and a global minimizer
of the original problem is obtained by reverting the permutations.

Before ending this chapter, for later use we quickly prove here the simple fact that $C_n$ is polyhedral,
and give its spanning directions.

\begin{lemma}
 The set $C_n \cap \R^n_{\ge 0}$ is a polyhedral cone.
 It is spanned by the vectors $b_1,\dots,b_n \in \R^n$ with%
 \footnote{It is obvious that any of the vectors $b_i$ can be scaled with an arbitrary
 positive number without changing the cone.  The scaling chosen here simplifies the notation
 in subsequent chapters.}
 \begin{equation}
 \label{eq:order_cone_spanning_directions}
  b_i
  \colonequals
  \begin{cases}
  \big( \underbrace{n,\dots,n}_{\text{$i$ many}},
  \underbrace{0,\dots,0}_{\text{$n-i$ many}} \big)^T  & \text{if $i<n$} \\
  \big( 1,\dots,1 \big)^T & \text{if $i=n$}.
  \end{cases}
 \end{equation}
\end{lemma}

Note also that the cone $C_n \cap \R^n_{\ge 0}$ is salient, i.e., it does not contain a linear subspace.

\begin{proof}
 Let $\operatorname{cone}(b_1,\dots,b_n) \subset \R^n$ denote the conical combination
 of the vectors $b_1,\dots,b_n$.  We first show that $C_n \cap \R^n_{\ge 0} \subset \operatorname{cone}(b_1,\dots,b_n)$.
 For this, let $x \in C_n \cap \R^n_{\ge 0}$.  Then
 \begin{equation}
  \label{eq:conical_combination}
  x = \sum_{i=1}^n x_i^c b_i
 \end{equation}
 with
 \begin{equation*}
  x_i^c
  \colonequals
  \begin{cases}
   x_i & \text{if $i=n$}, \\
   \frac{1}{n}(x_i - x_{i+1}) & \text{otherwise}.
  \end{cases}
 \end{equation*}
 By construction $x_1^c,\dots,x_n^c \ge 0$, and therefore
 \eqref{eq:conical_combination} is a valid conical combination.
 This proves $C_n \cap \R_{\ge 0}^n \subset \operatorname{cone}(b_1,\dots,b_n)$.

 To show the converse,
 note first that all $b_i$ are in $C_n \cap \R^n_{\ge 0}$, and that linear combinations
 with nonnegative coefficients of elements of $C_n \cap \R^n_{\ge 0}$ are in $C_n \cap \R^n_{\ge 0}$.
 Hence $\operatorname{cone}(b_1,\dots,b_n) \subset C_n \cap \R^n_{\ge 0}$.
\end{proof}

\subsection{Uniqueness and non-uniqueness in the order cone}
\label{sec:uniqueness_in_the_order_cone}

In the previous section we have shown that for any point $a \in C_n \cap \R^n_{\ge 0}$ there is
at least one point $p \in C_n \cap \sl(n) \cap \R^n_{\ge 0}$ that minimizes the squared Euclidean
distance $\dist{a}$ globally.
By strict convexity of $\sl^+(n) \cap \R^n_{\ge 0}$, uniqueness of that minimizer
follows directly if $\product(a) \le 1$.
This includes the case that $a$ has zero components, which happens if $A$ is not invertible.

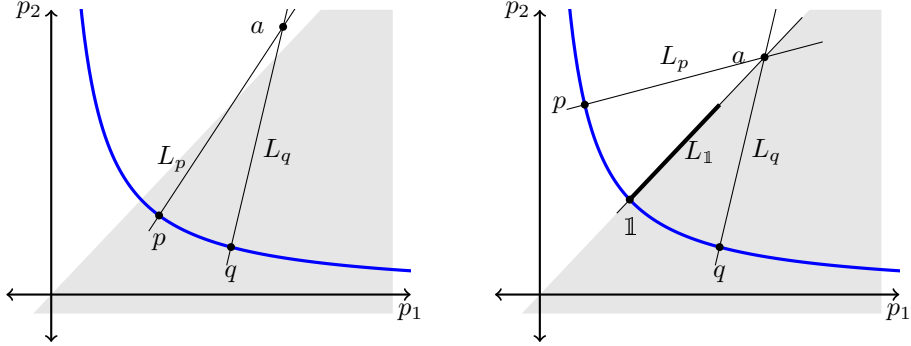
\begin{figure}
  \centering
  \begin{tikzpicture}

    % Background layer for the order cone
    \pgfdeclarelayer{background}
    \pgfsetlayers{background,main}

    \begin{axis}[
     axis x line=middle,
     axis y line=middle,
     axis line style={<->,color=black, thick},
     xlabel={$p_1$},
     ylabel={$p_2$},
     x label style={anchor=north},
     y label style={anchor=east},
     width=0.55\textwidth,
     xmin=-0.5,xmax=4,
     ymin=-0.5,ymax=3,
     xtick={0,5,...,},
     ytick={0,5,...,},
    ]

    % Draw the order cone
    \begin{pgfonlayer}{background}
     \fill [color=black!10!white] (-0.2,-0.2) -- (3.8,-0.2) -- (3.8,3.0) -- (3,3) -- cycle;
    \end{pgfonlayer}

    \clip(0.13,0) rectangle (5.29,3.27);
    \draw[blue, very thick,smooth,samples=100,domain=0.13409213855659688:5.285979884247383] plot(\x,{1/\x});

    \coordinate (p) at (1.2,1/1.2);
    % See https://github.com/pgf-tikz/pgfplots/issues/439 for why 'axis direction cs' is needed
    \coordinate (pNorm) at (axis direction cs:1/1.2, 1.2);

    \fill (p) circle (1.5pt);
    \node [label=below:{$p$}] at (p) {};
    \draw [shorten <= -0.25cm, shorten >= -2.0cm, name path=Lp] (p) -- +(pNorm) node [midway, left] {$L_p$};

    \coordinate (q) at (2,1/2);
    \coordinate (qNorm) at (axis direction cs:1/2, 2);

    \fill (q) circle (1.5pt);
    \node [label=below:{$q$}] at (q) {};
    \draw [shorten <= -0.25cm, shorten >= -1.0cm, name path=Lq] (q) -- +(qNorm) node [midway,right] {$L_q$};

    % TikZ is supposed to be able to compute the intersection 'a',
    % but I can't get this to work.
    % \path [name intersections={of=Lp and Lq,name=a}];
    % \node [fill=red,inner sep=1pt,label=-90:$a$] at (a-1) {};

    \coordinate (a) at (2.58,2.82);
    \fill (a) circle (1.5pt);
    \node [label=left:{$a$}] at (a) {};

   \end{axis}

  \end{tikzpicture}
  \qquad
  \begin{tikzpicture}

    % Background layer for the order cone
    \pgfdeclarelayer{background}
    \pgfsetlayers{background,main}

    \begin{axis}[
     axis x line=middle,
     axis y line=middle,
     axis line style={<->,color=black, thick},
     xlabel={$p_1$},
     ylabel={$p_2$},
     x label style={anchor=north},
     y label style={anchor=east},
     width=0.55\textwidth,
     xmin=-0.5,xmax=4,
     ymin=-0.5,ymax=3,
     xtick={0,5,...,},
     ytick={0,5,...,},
    ]

    % Draw the order cone
    \begin{pgfonlayer}{background}
     \fill [color=black!10!white] (-0.2,-0.2) -- (3.8,-0.2) -- (3.8,3.0) -- (3,3) -- cycle;
    \end{pgfonlayer}

    \clip(0.13,0) rectangle (5.29,3.27);
    \draw[blue, very thick,smooth,samples=100,domain=0.13409213855659688:5.285979884247383] plot(\x,{1/\x});

    \draw[black, ultra thick] (1,1) -- (2,2);
    %\draw[black, ultra thick, arrows={Bracket-Parenthesis}] (1,1) -- (2,2);  % Needs arrows.meta

    % The point p and its corresponding line
    \coordinate (p) at (1/2,2);
    % See https://github.com/pgf-tikz/pgfplots/issues/439 for why 'axis direction cs' is needed
    \coordinate (pNorm) at (axis direction cs:2, 1/2);

    \fill (p) circle (1.5pt);
    \node [label=left:{$p$}] at (p) {};
    \draw [shorten <= -0.25cm, shorten >= -0.8cm] (p) -- +(pNorm) node [midway, above] {$L_p$};

    % The point '1' and its corresponding line
    \coordinate (one) at (1,1);
    \coordinate (oneNorm) at (axis direction cs:1, 1);

    \fill (one) circle (1.5pt);
    \node [label=below:{$\mathbbm{1}$}] at (one) {};
    \draw [shorten <= -0.25cm, shorten >= -1.6cm] (one) -- +(oneNorm) node [midway,right] {$L_\mathbbm{1}$};

    % The point q and its corresponding line
    \coordinate (q) at (2,1/2);
    \coordinate (qNorm) at (axis direction cs:1/2, 2);

    \fill (q) circle (1.5pt);
    \node [label=below:{$q$}] at (q) {};
    \draw [shorten <= -0.25cm, shorten >= -1.0cm] (q) -- +(qNorm) node [midway,right] {$L_q$};

    \coordinate (a) at (2.5,2.5);
    \fill (a) circle (1.5pt);
    \node [label=left:{$a$}] at (a) {};

   \end{axis}

  \end{tikzpicture}

  \caption{Illustrations of the proofs of Lemmas~\ref{lem:uniqueness-n-2} and~\ref{lem:non_generic_case_2d},
   which both discuss the case $n=2$. Left: If two points $p$ and $q$ in the interior of the order cone (in grey)
   are both stationary for the squared distance to a point~$a$, then this point $a$ cannot also be
   in the order cone. Right: If $a_1 = a_2 > 2$, then there are exactly three stationary points.}
\end{figure}

For $\product(a) > 1$, the order cone can contain more than one stationary
point, and it is generally unclear how to single out the global minimizer
algorithmically.  An exception is the case $n=2$, where the stationary points
of the distance $\dist{a}$ can be classified completely.
We treat this case first.

\subsubsection{Uniqueness of the global minimizer in the order cone for $n=2$}

In the following lemma, $\mathring{C}_2$ denotes the interior
of the order cone.

\begin{lemma}\label{lem:uniqueness-n-2}
  Let $a\in C_2$, that is $a_1 \ge a_2 \ge 0$, and assume $\product(a) = a_1a_2 > 1$.
  Then there exists a unique stationary point $p\in \mathring{C}_2 \cap \sl(2) \cap \R_{\ge 0}^2$
  of the squared Euclidean distance to $a$.
\end{lemma}

\begin{proof}
  The proof is done by contradiction.
  Let $p\neq q\in \mathring{C}_2 \cap \sl(2)\cap \R_{\ge 0}^2$ both be solutions of the stationarity
  condition~\eqref{eq:lagrange_vector_formulation}.
  Then, by construction, the point $a$ is the intersection of the two lines
  \begin{equation}
  \label{eq:normal_lines}
    L_p \colonequals \big\{ p + \lambda p^{-1}:\ \lambda\in\R \big\},
    \qquad
    L_q \colonequals \big\{ q + \mu q^{-1} :\ \mu\in\R \big\}.
  \end{equation}
  Since $\nabla \product(x) = (x_1^{-1},x_2^{-1})^T$
  on $\sl(2)$,
  these lines are normal to the admissible set $\sl(2)$.
  The line $L_p$ intersects the boundary of the order cone at a single point $(p_1 + p_2,p_1+p_2)$,
  which corresponds to the value $\lambda = p_1p_2 = 1$.
  Likewise, $L_q$ intersects the boundary of the order cone
  at $(q_1+q_2,q_1+q_2)$ for $\mu = q_1q_2 = 1$.

  We now compute the values of $\lambda$ and $\mu$ that correspond to the intersection~$a$
  of the lines $L_p$ and $L_q$, and show that both are larger than~$1$, which contradicts
  the assumption $a \in C_2 \cap \R^2_{\ge 0}$.
  The values $\lambda$ and $\mu$ of the intersection solve the linear system
  \begin{equation*}
    \begin{pmatrix}
      -p_1^{-1} & q_1^{-1} \\
      -p_2^{-1} & q_2^{-1}
    \end{pmatrix}  \begin{pmatrix}
      \lambda \\ \mu
    \end{pmatrix} = \begin{pmatrix}
      p_1 - q_1 \\
      p_2 - q_2
    \end{pmatrix}.
  \end{equation*}
  There is a unique solution because $p \neq q$ and therefore the
  determinant of the matrix is nonzero.
  Using that $p_2 = p_1^{-1}$ and $q_2 = q_1^{-1}$ together with Cramer's rule we get the solution
  \begin{equation}\label{eq:lambda_mu}
    \lambda = \frac{ p_1q_1^3 + 1 }{ q_1( p_1 + q_1) },
    \qquad
    \mu =  \frac{ p_1^3q_1 + 1 }{ p_1( p_1 + q_1) }.
  \end{equation}
  To see now that $\lambda > 1$, multiply the left equation of~\eqref{eq:lambda_mu}
  by the denominator (which is positive) to obtain the equivalent condition
  \begin{equation*}
   p_1 q_1^3 + 1 > q_1 (p_1 + q_1).
  \end{equation*}
  This can be rewritten as
  \begin{align*}
   p_1 q_1 (q_1^2 - 1) + 1 - q_1^2
   =
   p_1 q_1 (q_1^2 - 1) - (q_1^2 - 1)
   =
   (p_1 q_1 - 1) (q_1^2 - 1)
   >
   0,
  \end{align*}
  which holds because both $p_1 > 1$ and $q_1 > 1$.
\end{proof}

For the non-generic case $a = (a_1, a_2)$ with $a_1 = a_2$ we can say
even more.

\begin{lemma}
\label{lem:non_generic_case_2d}
 Let $a = (a_1,a_2)$ such that $a_1 = a_2 \ge 0$.
 \begin{enumerate}
  \item \label{item:2d_unique_stationary_point}
  If $a_1 = a_2 \le 2$ then $\dist{a}$ has a unique stationary point $p = (1,1)^T$.

  \item \label{item:2d_pair_of_stationary_points}
  If $a_1 = a_2 > 2$ then there are two additional stationary points
   \begin{equation}
   \label{eq:2d_pair_of_stationary_points}
    p
    =
    \begin{pmatrix}
    \frac{a}{2} + \sqrt{\frac{a^2}{4} - 1} \\
    \frac{a}{2} - \sqrt{\frac{a^2}{4} - 1}
    \end{pmatrix}
    \qquad
    q = a - p
    =
    \begin{pmatrix}
    \frac{a}{2} - \sqrt{\frac{a^2}{4} - 1} \\
    \frac{a}{2} + \sqrt{\frac{a^2}{4} - 1}
    \end{pmatrix}.
   \end{equation}
  These are the only further ones in $\sl(2) \cap \R_{\ge 0}^2$,
  and they are global minimizers.
 \end{enumerate}
\end{lemma}

\begin{proof}
 \ref{item:2d_unique_stationary_point})
 Let $p \in \sl(2) \cap \R^2_{\ge 0}$ be a stationary point for the distance to
 the point $a$.  If $p = (1,1)^T$ then the normal line
 \begin{equation*}
  L_p \colonequals \big\{ p + \lambda p^{-1} \; : \; \lambda \in \R\big\}
 \end{equation*}
 is the diagonal axis, and hence all points $a$ with $a_1 = a_2$ have $(1,1)^T$
 as stationary point.

 If $p \neq (1,1)^T$ then the intersection of $L_p$ with the diagonal is at
 $(p_1 + p_2, p_1 + p_2)$ (for $\lambda = p_1 p_2$).  But since $p_1p_2 = 1$
 and $p_1, p_2 > 0$ we have $p_1 + p_2 = p_1 + \frac{1}{p_1} > 2$, and therefore
 no point on the diagonal with $a_1 = a_2 \le 2$ can have a second
 stationary point.

 \ref{item:2d_pair_of_stationary_points})
 For all points with $a_1 = a_2 > 2$,
 the pair of points~\eqref{eq:2d_pair_of_stationary_points}
 are stationary, as can be shown by checking the stationarity
 conditions~\eqref{eq:lagrange_vector_formulation} with $\lambda = 1$.
 Lemma~\ref{lem:uniqueness-n-2} implies that these are the only ones.

 To show that the two points of~\eqref{eq:2d_pair_of_stationary_points}
 are global minimizers, we show that $(1,1)^T$ is a local maximizer
 if $a_1 = a_2 > 2$.  The other two stationary points are then local
 minimizers by the coercivity of $\dist{a}$ on $\sl(2)$, and the local
 minimizers are global because they both have the same minimal value
 by symmetry of the problem.

 To show that $(1,1)^T$ is a local maximizer, parametrize $\sl(2)$
 using the Lie exponential map~\eqref{eq:exponential_map},
 and write the squared distance as
 \begin{equation*}
  \dist{\exp}(t)
  \colonequals
  \dist{a}(\exp (t,-t)^T)
  =
  \frac{1}{2} (\exp t - a_1)^2 + \frac{1}{2} (\exp (-t) - a_2)^2.
 \end{equation*}
 Since $a_1 = a_2$, its second derivative is
 \begin{equation*}
  \dist{\exp}''(t)
  =
  2 \big(\exp(2t) + \exp(-2t)\big) - a_1\big(\exp t + \exp (-t)\big).
 \end{equation*}
 At $t=0$ this is $4-2a_1$, which is negative if and only if $a_1>2$.
\end{proof}

\subsubsection{The case $n \ge 3$}
\label{sec:nonuniqueness}

In higher dimensions, the order cone can contain more than one stationary point.
To show this we introduce a particular path, which will later also serve
as the basis of one numerical algorithm. Let $a$ be a given point in $C_n \cap \R^n_{\ge 0}$.

\begin{definition}[Solution path]\label{def:solution_path}
 The positive and negative branches of the solution path are
 \begin{alignat*}{2}
  \mathcal{P}^+(\lambda) & : \big(-\infty, \tfrac{a_n^2}{4}\big] \to \R^n
  & \qquad
  \mathcal P^+(\lambda)_i & \coloneqq \tfrac{a_i}{2} + \sqrt{ \tfrac{a_i^2}{4} - \lambda},
  \qquad
  i=1,\hdots,n
  \shortintertext{and}
  \mathcal{P}^-(\lambda) & : \big(0, \tfrac{a_n^2}{4}\big] \to \R^n
  &
  \mathcal P^-(\lambda)_i
  & \coloneqq
  \begin{cases}
   \frac{a_i}{2} + \sqrt{ \frac{a_i^2}{4} - \lambda} & \text{if $i=1,\hdots,n-1$}, \\
   \frac{a_i}{2} - \sqrt{ \frac{a_i^2}{4} - \lambda} & \text{if $i=n$},
  \end{cases}
 \end{alignat*}
 respectively. Combine the two paths into a single one by
  \begin{equation*}
    \mathcal P : \big(-\infty ,\tfrac{a_n^2}{2}\big] \to \R^n,
    \qquad
    \mathcal P(\lambda)
    \coloneqq
    \begin{cases}
      \mathcal P^+(\lambda) & \text{if $\lambda < \tfrac{a_n^2}{4}$} \\
      \mathcal P^-\big(\tfrac{a_n^2}{2}-\lambda\big) & \text{if $\lambda \ge \tfrac{a_n^2}{4}$}.
    \end{cases}
  \end{equation*}
\end{definition}
By construction, no negative numbers appear under the square root,
and the path is well-defined.
Figure~\ref{fig:solution-path} shows an example path for $n=2$.
The following lemma shows the relevant properties of the path.

\begin{figure}
  \centering
 \begin{tikzpicture}

 % Background layer for the order cone
 \pgfdeclarelayer{background}
 \pgfsetlayers{background,main}

 \begin{axis}[
         axis x line=middle,    % put the x axis in the middle
         axis y line=middle,    % put the y axis in the middle
         axis line style={<->,color=black, thick}, % arrows on the axis
         xmin=-0.5,xmax=3.5,
         ymin=-0.5,ymax=3,
         xtick={0,1,...,3},
         ]

         % Draw the order cone
         \begin{pgfonlayer}{background}
          \fill [color=black!10!white] (0,0) -- (3.3,0) -- (3.3,3.3) -- cycle;
         \end{pgfonlayer}

         %\fill (0,0) -- (1.0, 0.0) -- (1.0, 1.0) -- cycle;
         % sl(2)
         \addplot [domain=0:4,samples=50,thick,color=black]{1/x};
         \node at (1.0,1.5) {$\sl(n)$};

         \def \ax{2.5}
         \def \ay{2.0}
         \coordinate (a) at (\ax, \ay);

         % \mathcal{P}^+
         \addplot [domain=-2.7:{\ay^2/4},samples=50,thick,color=green]({\ax/2 + sqrt(\ax^2/4 - x)},{\ay/2 + sqrt(\ay^2/4 - x)});
         \node at (2.3,1.35) {$\mathcal{P}^+$};

         % \mathcal{P}^-
         \addplot [domain=0:{\ay^2/4},samples=50,thick,color=red]({\ax/2 + sqrt(\ax^2/4 - x)},{\ay/2 - sqrt(\ay^2/4 - x)});
         \node at (2.3,0.65) {$\mathcal{P}^-$};

         % The two important points
         \draw[fill] (\ax,\ay) circle (3pt);
         \node at (\ax+0.15,\ay-0.15) {$a$};

         \draw[fill] (\ax,0.0) circle (3pt);
         \node at (\ax+0.15,0.15) {$\bar a$};
 \end{axis}
 \end{tikzpicture}
  \caption{The solution path for \(a=(2.5, 2)^T\)}
  \label{fig:solution-path}
\end{figure}
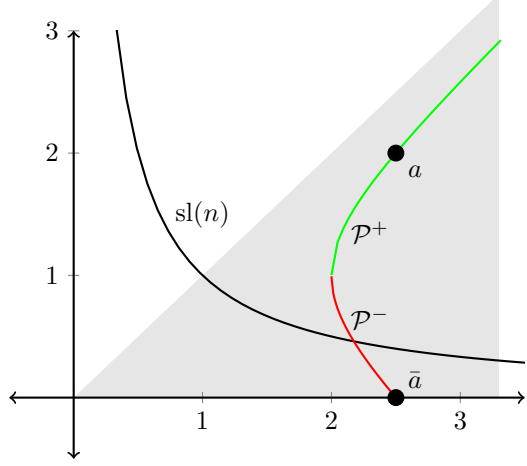

\begin{lemma} \mbox{}
\label{lem:solution_path_properties}
 \begin{enumerate}[label=(A\arabic*)]
  \item \label{item:solution_path_continuity}
  The solution path is continuous.

  \item \label{item:solution_path_ends}
  $\mathcal{P}(0) = a$ and $\mathcal{P}\big(\frac{a_n^2}{2}\big) = \bar{a} \colonequals (a_1,\dots,a_{n-1},0)$.

  \item \label{item:solution_path_ordering}
  $\mathcal{P}(\lambda) \in C_n$ for all $\lambda \in \big(-\infty,\frac{a_n^2}{2}\big]$.

  \item \label{item:solution_path_solutions}
  For any $\lambda \in \big(-\infty,\frac{a_n^2}{2}\big]$ with $\product \big(\mathcal{P}(\lambda) \big) = 1$,
  the point $\mathcal{P}(\lambda)$ is a stationary point of the squared distance $\dist{a}$.

  \item \label{item:solution_path_monotony}
  For $\lambda \le 0$, the map $\lambda \mapsto \product \big(\mathcal{P}(\lambda) \big)$
  is strictly monotone decreasing, and $\lim_{\lambda \to -\infty} \product \big(\mathcal{P}(\lambda)\big) = \infty$.
 \end{enumerate}
\end{lemma}

\begin{proof}
 Assertions~\ref{item:solution_path_continuity} and \ref{item:solution_path_ends}
 are easily checked.

 To show~\ref{item:solution_path_ordering}, consider first
 \begin{equation*}
  \mathcal{P}^+(\lambda)_i = \frac{a_i}{2} + \sqrt{\frac{a_i^2}{4} - \lambda}.
 \end{equation*}
 This expression is strictly monotone increasing in $a_i$, and therefore
 $a_i \ge a_j$ implies $\mathcal{P}^+(\lambda)_i \ge \mathcal{P}^+(\lambda)_j$.
 Hence $\mathcal{P}^+(\lambda) \in C_n$ for all $\lambda \in \big(-\infty,\frac{a_n^2}{4}\big]$.

 To show that also $\mathcal{P}^-(\lambda) \in C_n$ for all
 $\lambda \in [0,\frac{a_n^2}{4}]$, note that
 the first $n-1$ coefficients of $\mathcal{P}^-$ coincide with those of $\mathcal{P}^+$.
 Therefore, all we need to check is that $\mathcal{P}^-_{n-1} \ge \mathcal{P}_n^-$,
 which follows from
 \begin{equation*}
  \mathcal{P}^-(\lambda)_{n-1}
  =
  \frac{a_{n-1}}{2} + \sqrt{\frac{a^2_{n-1}}{4} - \lambda}
  \ge
  \frac{a_n}{2} + \sqrt{\frac{a^2_n}{4} - \lambda}
  \ge
  \frac{a_n}{2} - \sqrt{\frac{a^2_n}{4} - \lambda}
  =
  \mathcal{P}^-(\lambda)_n.
 \end{equation*}

 For~\ref{item:solution_path_solutions}, note that
 by construction the points $\mathcal{P}(\lambda)$ solve
 \begin{equation*}
  \mathcal{P}(\lambda)_i + \lambda \mathcal{P}(\lambda)^{-1}_i = a_i
  \qquad
  \forall i=1,\dots,n,
 \end{equation*}
 which is the first part of the stationarity condition~\eqref{eq:lagrange_vector_formulation}.
 If in addition $\product \big( \mathcal{P}(\lambda) \big) = 1$, then the second part also holds,
 and assertion~\ref{item:solution_path_solutions} follows.

 To prove~\ref{item:solution_path_monotony}, note that for any number $\alpha \ge 0$, the map
 \begin{equation*}
  z_\alpha(\lambda)
  \colonequals
  \frac{\alpha}{2} + \sqrt{\frac{\alpha^2}{4} - \lambda}
 \end{equation*}
 is defined for all $\lambda \le 0$, and it is
 strictly monotone decreasing.  Moreover, $\lim_{\lambda \to -\infty} z_\alpha(\lambda) = \infty$.
 Since also the functions $z_{a_i},\dots,z_{a_n}$ are nonnegative whenever they are defined,
 the map
 \begin{equation}
 \label{eq:determinant_as_function_of_lambda}
  \lambda
  \mapsto
  \prod_{i=1}^n z_{a_i}(\lambda)
  =
  \prod_{i=1}^n \Big(\frac{a_i}{2} + \sqrt{\frac{a_i^2}{4} - \lambda}\Big)
 \end{equation}
 is defined for all $\lambda \le 0$, is strictly monotone decreasing,
 and approaches $\infty$ as $\lambda \to - \infty$.
\end{proof}

The continuity of $\mathcal{P}$ allows to show existence of stationary points.
If $\product(a) <1$ then, by \ref{item:solution_path_monotony}, $\mathcal{P}$ intersects
$\sl(n)$ exactly once (which is not surprising, as this case
corresponds to the projection onto a strictly convex set).
If $\product(a) > 1$, then the restriction of $\mathcal{P}$ to $\big[0,\frac{a_n^2}{2}\big]$
is a continuous path in $C_n$ from~$a$ (with $\product(a) > 1$) to the point
$\bar{a} \coloneqq (a_1,\dots,a_{n-1},0)$ with $\product(\bar{a}) = 0$.
Therefore, there must be at least
one intersection of $\mathcal{P}$ with $\sl(n)$, i.e., at least
one stationary point of $\dist{a}$ in $C_n \cap \sl(n) \cap \R^n_{\ge 0}$.

Besides forming the basis of a numerical algorithm in Chapter~\ref{sec:direct_root_finding},
the solution path is also helpful to illustrate that there can be more than one
stationary point of $\dist{a}$ in $C_n$ if $n \ge 3$.
Figure~\ref{fig:solution-path-product} shows such a situation for $n=3$,
where the solution path intersections $\sl(3)$ three times in the order cone,
corresponding to three solutions of the stationarity equations~\ref{eq:lagrange_vector_formulation}.
It is unclear which of the stationary points (if any) is the global minimizer.

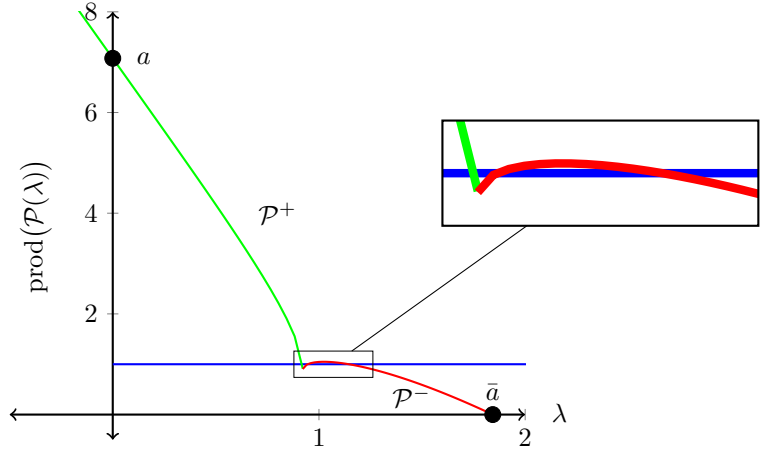
\begin{figure}
  \centering
  \begin{tikzpicture}[spy using outlines={magnification=4, size=4.2cm, height=1.4cm, connect spies}]

   \begin{axis}[
    axis x line=middle,
    axis y line=middle,
    axis line style={<->,color=black, thick},
    xlabel={$\lambda$},
    ylabel={$\textup{prod}\big(\mathcal{P}(\lambda)\big)$},
    x label style={at={(axis description cs:1.1,0.02)}},
    y label style={at={(axis description cs:0.1,.5)},rotate=90,anchor=south},
    xmin=-0.5,xmax=2,
    ymin=-0.5,ymax=8,
    xtick={0,1,...,3},
    ]

    \addplot [domain=0:4,samples=50,thick,color=blue] {1};

    \def \ax{1.92}
    \def \ay{1.9199}
    \def \az{1.9198}

    % \product( \mathcal{P}^+(\lambda) )
    \addplot [domain=-1:{\az^2/4},samples=50,thick,color=green] {(\ax/2 + sqrt(\ax^2/4 - x)) * (\ay/2 + sqrt(\ay^2/4 - x)) * (\az/2 + sqrt(\az^2/4 - x))};
    \node at (0.8,4.0) {$\mathcal{P}^+$};

    % \product( \mathcal{P}^-(\lambda) )
    \addplot [domain=0:{\az^2/4},samples=50,thick,color=red]({\az^2/2 - x}, {(\ax/2 + sqrt(\ax^2/4 - x)) * (\ay/2 + sqrt(\ay^2/4 - x)) * (\az/2 - sqrt(\az^2/4 - x))});
    \node at (1.45,0.35) {$\mathcal{P}^-$};

    % The two important points
    \coordinate (aImg) at (0,\ax*\ay*\az);
    \draw[fill] (aImg) circle (3pt);
    \node at (0.15,\ax*\ay*\az) {$a$};

    \draw[fill] (\az^2/2,0.0) circle (3pt);
    \node at (\az^2/2,0.45) {$\bar a$};

    % Spy center
    \path (1.07,1) coordinate (spyCenter);

   \end{axis}

   \spy [black] on (spyCenter) in node [above right] at ([xshift=-1.1cm]current axis.east);

  \end{tikzpicture}
  \caption{The expression $\product\big(\mathcal{P}(\lambda)\big)$ as a function of $\lambda$,
  for \(a=(1.92, 1.9199, 1.9198)^T\)}
  \label{fig:solution-path-product}
\end{figure}

\section{Coordinate transformations}
\label{sec:coordinate_transformations}

One aspect that makes solving the projection problem onto $\sl(n) \cap \R^n_{\ge 0}$ difficult
is the nonlinearity of the admissible set. In this chapter we therefore explore
coordinate transformations that simplify this set.

\subsection{Logarithmic coordinates}
\label{sec:logarithmic_coordinates}

As a first step, we transform the problem such that the admissible set becomes
a vector space.  For this, we make the substitution
\begin{equation*}
 \xi_i \colonequals \log x_i,
 \qquad
 i = 1,\dots,n.
\end{equation*}
This induces a diffeomorphism between the sets $\R^n_{>0}$ and $\R^n$.
In the new coordinates, the constraint $\product(x) = 1$ transforms to the linear condition
\begin{equation}
\label{eq:linearized_constraint}
 \sum_{i=1}^n \xi_i
 =
 \log \prod_{i=1}^n x_i
 =
 \log 1
 =
 0.
\end{equation}
In the following, we will call this admissible set
\begin{equation*}
 V
 \colonequals
 \Big\{ \xi \in \R^n \; : \; \sum_{i=1}^n \xi_i = 0 \Big\}.
\end{equation*}
Note that it is parallel to the tangent space $T_{\mathbbm{1}} \sl(n)$ at
the point $\mathbbm{1} \colonequals (1,\dots,1)^T$.

Incidentally, the logarithmic transformation also preserves the order cone.
This follows from the strict monotonicity of the logarithm,
together with the fact that $\log$ is a diffeomorphism from $\R_{>0}$ to $\R$.

\begin{lemma}
 The transformation $\xi = \log x$ (componentwise) is a diffeomorphism between the strictly positive
 order cone $C_n \cap \R_{>0}^n$ and $C_n$ itself.
\end{lemma}

We will sometimes write $C_n^\textup{log}$ instead of $C_n$ to emphasize the fact
that logarithmic coordinates are used.

The cone $C_n^\textup{log}$ is still polyhedral, but it is
not salient anymore, because it contains the linear space
\begin{equation*}
 \operatorname{span}(b_n)
 =
 \operatorname{span}(\mathbbm{1})
 =
 \operatorname{span}\big((1,\dots,1)^T\big).
\end{equation*}
Recalling definition~\eqref{eq:order_cone_spanning_directions}
of the spanning directions $b_1,\dots,b_n$, the cone can therefore be written as
\begin{equation}
\label{eq:logarithmic_cone}
 C_n^\textup{log}
 =
 \cone(b_1,\dots,b_{n-1})
 +
 \operatorname{span}(\mathbbm{1}).
\end{equation}

\subsection{The squared Euclidean distance in logarithmic coordinates}
\label{sec:logarithmic_distance}

Of course, linearizing the admissible set $\sl(n)$
has its price---the functional to be minimized ceases to be quadratic.

In logarithmic coordinates, the squared distance functional is
\begin{align}
\label{eq:distance_in_log_coordinates}
 \distlog{a}  : \R^n \to \R
 \qquad
 \distlog{a}(\xi)
 \colonequals
 \frac{1}{2} \sum_{i=1}^n \big( \exp \xi_i - a_i \big)^2.
\end{align}
It is defined on the entire space $\R^n$, and like the Cartesian squared distance $\dist{a}$
it is bounded from below by zero. In contrast to $\dist{a}$, however,
$\distlog{a}$ is not coercive on $\R^n$, because
$\distlog{a}(\xi^k)$ converges to $\frac{1}{2} \sum_{i=1}^n a_i^2 < \infty$
for any sequence $(\xi^k)$ converging to $(-\infty, \dots, -\infty)$.
Fortunately, all that matters is coercivity on the
admissible set.

\begin{lemma}
 The functional $\distlog{a}$ is coercive on $V$, i.e., for each sequence $(\xi^k)$ in $V$
 with $\norm{\xi^k} \to \infty$ we get $\distlog{a}(\xi^k) \to \infty$.
\end{lemma}

\begin{proof}
 We show the assertion for the norm $\norm{\xi}_{\infty} \colonequals \max_{i=1,\dots,n} \abs{\xi_i}$.
 Let $\xi^k$ be a sequence in~$V$ with $\norm{\xi^k}_{\infty} \to \infty$. Then for each $k$
 there is an index $i_k \in \{1,\dots,n\}$ such that $\norm{\xi^k}_{\infty} = \abs{\xi^k_{i_k}}$.
 Since $\xi_1^k + \dots + \xi_n^k = 0$, there is then also a $j_k \in \{1,\dots,n\}$
 with $\xi^k_{j_k} \ge \frac{1}{n-1} \abs{\xi^k_{i_k}}$.
 Then
 \begin{equation*}
  \distlog{a}(\xi^k)
  \ge
  \frac{1}{2} \big(\exp \xi^k_{j_k} - a_{j_k} \big)^2
 \end{equation*}
 because each addend of~\eqref{eq:distance_in_log_coordinates} is non-negative, and
 \begin{equation*}
  \big(\exp \xi^k_{j_k} - a_{j_k} \big)^2
  \ge
  \Big( \exp \frac{\abs{\xi^k_{i_k}}}{n-1} - a_{j_k} \Big)^2
 \end{equation*}
 for $k$ large enough,
 because the scalar map $\tau \mapsto (\exp \tau - \text{const})^2$ is monotone increasing for all $\exp \tau > \text{const}$.
 But $\norm{\xi^k}_{\infty} \to \infty$ is the same as $\abs{\xi_{i_k}} \to \infty$,
 and therefore this implies $\distlog{a}(\xi^k) \to \infty$.
\end{proof}

Before we continue, we briefly introduce the function
\begin{equation*}
 E : \R^n \supset V \to \R,
 \qquad
 E(\xi) \colonequals \sum_{i=1}^n \exp \xi_i,
\end{equation*}
which will appear several times in the second half of the paper.
Because of $E(\xi) = \langle \exp \xi, \mathbbm{1} \rangle$, the function can be interpreted
as the vertical distance between $V$ and the image of the Lie exponential map
$\exp \xi \subset \sl(n)$. At the same time, $E$ is one possible way to generalize
the hyperbolic cosine function to higher dimensions.

\begin{lemma}
\label{lem:hyperbolic_cosine_is_convex}
 The functional $E$ restricted to $V$ is strictly convex, coercive, and has a
 unique minimizer at $\xi=0$.
\end{lemma}
\begin{proof}
 The functional is strictly convex, because its Hesse matrix
 \begin{equation*}
  \nabla^2 E(\xi) = \diag \big( \exp \xi_1, \dots, \exp \xi_n \big)
 \end{equation*}
 is positive definite everywhere, in particular on $V$.
 Restricted to $V$, it has a minimizer at $\xi=0$, because $E(0) = n$, but at the
 same time for all $\xi \in V$ we have
 \begin{equation*}
  E(\xi) = \sum_{i=1}^n \exp \xi_i \ge \sum_{i=1}^n (\xi_i + 1) = n.
 \end{equation*}

 To show coercivity of $E$ on $V$, let $S$ be the unit sphere in $V$ around $0$.
 Since this sphere is compact and $E$ is continuous, $E$ assumes its minimum
 $E_{\text{min},S}$ on $S$, and this minimum is strictly larger than $n$.
 Then, by convexity, along every ray through~$0$, $E$ grows at least affinely,
 with slope $E_{\text{min},S} - n$. This implies coercivity.
 Coercivity together with strict convexity implies that the minimizer
 at $\xi = 0$ is unique.
\end{proof}

\subsection{Hyperbolic coordinates}

So far, we have replaced the nonlinear admissible set $\sl(n)$ by the vector space~$V$.
We now do a further transformation that eliminates the constraint
completely. For this we first note that because of
$\sum_{i=1}^n \xi_i = \langle \xi, \mathbbm{1}\rangle$ for every admissible~$\xi$
in logarithmic coordinates,
the admissible set $V$ in such coordinates
is the orthogonal complement of the linear subspace $\operatorname{span}(\mathbbm{1})$
of $C_n^\textup{log}$. We can eliminate
the constraint if we can construct a basis for this complement.
At the same time, we want to retain a simple description of the order cone
with respect to that new basis, because we know that this cone contains at least
one global minimizer of $\dist{a}$.  This is indeed possible.  In fact,
in the new hyperbolic coordinates, the restriction of the order cone $C^{\textup{log}}_n$
to the admissible set~$V$ turns into the positive orthant of $\R^{n-1}$!
We achieve this
by modifying the set of vectors $b_1,\dots,b_{n-1}$ that span the salient part of $C_n^\textup{log}$,
because the representation of $C_n^\textup{log}$ given in~\eqref{eq:logarithmic_cone}
is not unique.  Indeed, since the cone contains the linear space $\operatorname{span}(\mathbbm{1})$
we can add multiples of $\mathbbm{1}$
to any of the other spanning vectors without changing the cone, i.e.,
\begin{equation*}
 C_n^\textup{log}
 =
 \cone(b_1 + \eta_1 \mathbbm{1},\dots,b_{n-1}+\eta_{n-1}\mathbbm{1})
 +
 \operatorname{span}(\mathbbm{1})
\end{equation*}
for any $\eta_1,\dots,\eta_{n-1} \in \R$.
We use this freedom to construct a new spanning set $\tilde{b}_i \colonequals b_i + \eta_i \mathbbm{1}$,
$i=1, \dots, n-1$ such that $\tilde{b}_i \in V$.  These $\tilde{b}_i$ then form a basis
of~$V$.

\begin{lemma}
 For every $i=1,\dots,n-1$, the vector
 $\tilde{b}_i \colonequals b_i + \eta_i \mathbbm{1}$ is in $V$ if $\eta_i = -i$.
\end{lemma}
\begin{proof}
By definition of $V$, the inclusion $\tilde{b}_i \in V$ is equivalent to
$\langle \tilde{b}_i, \mathbbm{1} \rangle = 0$.
With the ansatz $\tilde{b}_i = b_i + \eta_i \mathbbm{1}$ we see that for any $i=1, \dots, n-1$
\begin{align*}
 \big \langle b_i + \eta_i \mathbbm{1}, \mathbbm{1} \big\rangle
 & =
 \Big\langle (\underbrace{n,\dots,n}_{\text{$i$ many}}, 0, \dots, 0)^T + \eta_i \mathbbm{1}, \mathbbm{1} \Big\rangle
 =
 in + \eta_i n.
\end{align*}
Hence $\langle \tilde{b}_i, \mathbbm{1} \rangle = 0$ if and only if
\begin{equation*}
 \eta_i = -i.
 \qedhere
\end{equation*}
\end{proof}

From now on the symbols $\tilde{b}_i$ will always represent this choice of $\eta_i$, i.e.,
\begin{equation*}
 \tilde{b}_i
 \colonequals
 \begin{cases}
  b_i - i\mathbbm{1} & \text{if $i<n$}, \\
  \mathbbm{1}        & \text{if $i=n$}.
 \end{cases}
\end{equation*}
Any point $\xi$ in the order cone $C_n^\textup{log}$ can then be represented as
\begin{equation}
\label{eq:hyperbolic_coordinates_representation}
 \xi
 =
 \sum_{i=1}^n \zeta_i \tilde{b}_i
 =
 \sum_{i=1}^{n-1} \zeta_i \tilde{b}_i + \zeta_n \mathbbm{1},
\end{equation}
with $\zeta_1,\dots,\zeta_{n-1} \ge 0$ und $\zeta_n \in \R$.
We call the numbers $\zeta_1,\dots,\zeta_n$ \emph{hyperbolic coordinates}.

\begin{remark}[The two-dimensional case]
\label{rem:2d_hyperbolic_coordinates}
 If $n=2$ we obtain
 \begin{equation*}
  \tilde{b}_1
  =
  \begin{pmatrix} 1 \\ -1 \end{pmatrix}
  \qquad \text{and} \qquad
  \tilde{b}_2
  =
  \begin{pmatrix} 1 \\ 1 \end{pmatrix}.
 \end{equation*}
 That is, a point $\zeta = (\zeta_1, \zeta_2)$ in the coordinates~\eqref{eq:hyperbolic_coordinates_representation}
 corresponds to
 \begin{equation*}
  \begin{pmatrix} x_1 \\ x_2 \end{pmatrix}
  =
  \begin{pmatrix} \exp \xi_1 \\ \exp \xi_2 \end{pmatrix}
  =
  \begin{pmatrix}
   \exp(\zeta_1 + \zeta_2) \\
   \exp(-\zeta_1 + \zeta_2)
  \end{pmatrix}
  =
  \exp \zeta_2
  \begin{pmatrix}
   \exp \zeta_1 \\
   \exp (-\zeta_1)
  \end{pmatrix}
 \end{equation*}
 in Cartesian coordinates. The numbers
 \begin{alignat}{2}
  \label{eq:2d_geometric_mean}
  \exp \zeta_2
  & =
  \sqrt{x_1 x_2}
  & \qquad&
  \text{(the geometric mean)}
  \shortintertext{and}
  \nonumber
  \zeta_1
  & =
  \log \sqrt{\frac{x_1}{x_2}}
  &&
  \text{(the hyperbolic angle)}
 \end{alignat}
 form coordinates of the positive quadrant, sometimes called hyperbolic coordinates.
 This justifies calling the coordinates $\zeta_1,\dots,\zeta_n$
 of the positive $n$-dimensional orthant defined by~\eqref{eq:hyperbolic_coordinates_representation} hyperbolic as well.
\end{remark}

For later reference we also need the inverse to~\eqref{eq:hyperbolic_coordinates_representation}.
Denote by $B \colonequals (\tilde{b}_1 \; | \; \tilde{b}_2 \; | \dots | \; \tilde{b}_n)$
the matrix of spanning vectors
\begin{align}
\label{eq:hyperbolic_transformation_matrix}
 B
 & =
 \begin{pmatrix}
  n-1 & n-2 & n-3 & \dots & 1 \\
   -1 & n-2 & n-3 & \dots & 1 \\
   -1 &  -2 & n-3 & \dots & 1 \\
    \vdots  & \vdots        & \vdots        &       & \vdots \\
  -1 &  -2 & -3 & \dots & 1
 \end{pmatrix},
\end{align}
such that~\eqref{eq:hyperbolic_coordinates_representation} can be written as $\xi = B\zeta$.
The following assertion can be derived with the Sherman--Morrison formula,
using the representation of $B$ as a rank-one update
\begin{align*}
 B
 & =
 (b_1 \; | \; b_2 \; | \dots | \; b_{n-1} \; | \; b_n)
 + \begin{pmatrix} 1 \\ \vdots \\ 1 \end{pmatrix}
 \big(-1 \; \; -2 \; \dots \; -(n-1) \; \; 0\big).
\end{align*}

\begin{lemma}
\label{lem:inverse_hyperbolic_coordinates}
 The matrix $B$ is invertible and
 \begin{equation*}
  B^{-1}
  =
  \frac{1}{n}
  \begin{pmatrix}
   1 & -1 &  0 & \dots & 0 \\
   0 &  1 & -1 & \ddots     & \vdots \\
   \vdots & \ddots & \ddots & \ddots & 0 \\
   0 & \dots  &     0 & 1 & -1 \\
   1 & 1 & \dots & 1 & 1
  \end{pmatrix}.
 \end{equation*}
\end{lemma}

Hence, the hyperbolic coordinates~$\zeta$ of a point $\xi$ in logarithmic
coordinates are given by $\zeta = B^{-1}\xi$.
As a direct corollary we obtain that
\begin{equation}
\label{eq:constraint_as_dof}
 \zeta_n = \frac{1}{n} \sum_{i=1}^n \xi_i.
\end{equation}
This implies that $B^{-1}$ bijectively maps the admissible set~$V$ onto
$\R^{n-1} \times \{ 0 \}$.  We have effectively eliminated the equality constraint.

\begin{remark}
 Numerical tests suggest that the condition number of \(B\)
 with respect to the spectral norm is \(\approx \frac{1}{3} n^{\frac{3}{2}}\).
% Here is how we came to this conclusion:
% First, compute the condition numbers by
%
% julia> for n in [2,3,4,8,16,32,64,128,256]
% B = ProjectionOntoSLn.B(n,Int64)
% println("$n  $(cond( [ B ones(n) ] ))")
% end
%
% Then print them with
% import matplotlib.pyplot as plt
% import numpy
%
% plt.plot([2, 3, 4, 8, 16, 32, 64, 128, 256],
%          [   1.0000000000000004,
%              1.7320508075688774,
%              2.613125929752753,
%              7.249019570823101,
%             20.40459447475665,
%             57.6433907560547,
%            162.99102533785143,
%            460.97352408337565,
%           1303.8054750756112],
%          'ro')
%
% # evenly sampled time
% t = numpy.arange(1.0, 256, 1)
% plt.plot(t, 0.33*t**1.5)
%
% plt.axis((1, 256, 1.0, 1400))
% plt.xscale('log')
% plt.yscale('log')
% plt.show()
\end{remark}

\subsection{The squared Euclidean distance in hyperbolic coordinates}
\label{sec:hyperbolic_distance}

We continue our investigation of the squared distance functional in non-Cartesian coordinates.
Recall Definition~\eqref{eq:distance_in_log_coordinates} of the functional
$\distlog{a}$ in logarithmic coordinates.

To transform $\distlog{a}$ to hyperbolic coordinates~$\zeta$
we use the relationship $\xi = B\zeta$,
with the matrix $B$ defined in~\eqref{eq:hyperbolic_transformation_matrix}.
Logarithmic-coordinate components $\xi_i$ can be expressed as
\begin{equation*}
 \xi_i = b_{i*} \zeta,
\end{equation*}
where $b_{i*}$ is the $i$-th row of the matrix $B$.
Since every row of $B$ ends with a~1 we can rewrite this as
\begin{equation}
\label{eq:hyperbolic_components}
 \xi_i = \bar{b}_{i*} \bar{\zeta} + \zeta_n,
\end{equation}
where we have used bars to denote the vectors of the first $n-1$ entries
of vectors in $\R^n$.
Inserting \eqref{eq:hyperbolic_components} into~\eqref{eq:distance_in_log_coordinates} leads to
\begin{align}
\label{eq:distance_in_hyp_coordinates}
 \disthyp{a}(\zeta)
 & \colonequals
 \frac{1}{2} \sum_{i=1}^n \big(\exp(\bar{b}_{i*} \bar{\zeta} + \zeta_n)
  - a_i
 \big)^2,
\end{align}
but the term $+\zeta_n$ can be omitted because $\disthyp{a}$ is to be
minimized in the space $\R^{n-1} \times \{ 0 \} = \big\{ \zeta \in \R^n\; :\; \zeta_n = 0 \big\}$ only.

The squared distance functional~\eqref{eq:distance_in_hyp_coordinates}
is still smooth and bounded from below. It is coercive
on $\R^{n-1} \times \{ 0 \}$ because the coordinate transformation $B$
between logarithmic and hyperbolic coordinates is an invertible linear
map between the two finite-dimensional spaces $V$ and $\R^{n-1} \times \{ 0 \}$.

However, looking at simple examples like the ones in Figures~\ref{fig:da_hyp_2d}
and~\ref{fig:da_hyp_3d} immediately shows that $\disthyp{a}$ is not necessarily convex anymore
if $\product(a) > 1$.  This is not surprising, as we know from Chapter~\ref{sec:nonuniqueness}
that there can be more than one minimizer even in the order cone.
We have a somewhat weaker structure, though: $\disthyp{a}$ is the difference
of two convex functions.%
\footnote{Knowing this it is possible to minimize $\disthyp{a}$ using
algorithms from the field of difference-of-convex programming~\cite{deoliveira:2020}.
This direction is not further pursued here, though.
}
To see this, introduce hyperbolic coordinates $\omega = (\bar{\omega}, \omega_n)$ for $a$ as well, i.e., set
\begin{equation*}
 a_i = \exp (\bar{b}_{i*} \bar{\omega} + \omega_n).
\end{equation*}
Inserting this into~\eqref{eq:distance_in_hyp_coordinates}, multiplying out the squares,
and fixing $\zeta_n = 0$ leads to
\begin{align*}
 \disthyp{\omega} : \R^{n-1} \to \R
 \qquad
 \disthyp{\omega}(\bar{\zeta})
 \colonequals
 \frac{1}{2} \sum_{i=1}^n \exp(2 \bar{b}_{i*} \bar{\zeta})
   - \exp \omega_n \sum_{i=1}^n \exp (\bar{b}_{i*} ( \bar{\omega} + \bar{\zeta}))
\end{align*}
plus an irrelevant constant term.
Introducing the reduced matrix $\bar{B} \in \R^{n \times (n-1)}$
whose rows are the vectors $\bar{b}_{i*}$, we can write this more concisely as
\begin{equation}
\label{eq:disthyp_reduced}
 \disthyp{\omega}(\bar{\zeta})
 =
 \frac{1}{2} E(2\bar{B} \bar{\zeta}) - \exp \omega_n E\big(\bar{B}(\bar{\omega} + \bar{\zeta})\big) + \text{const},
\end{equation}
where $E(x) \colonequals \sum_{i=1}^n \exp x_i$
is the map already seen in Chapter~\ref{sec:logarithmic_distance}.
This really is the difference between two convex functionals,
because from Lemma~\ref{lem:hyperbolic_cosine_is_convex} we get:
\begin{lemma}
 The functional $E^\textup{hyp}$ defined by $\bar{\zeta} \mapsto E(\bar{B} \bar{\zeta})$
 is strictly convex, coercive, and has a unique minimizer at $\bar{\zeta} = 0$.
\end{lemma}

\begin{figure}
 \begin{center}
  \subfigure[Convex: $a = (2,\frac{1}{4})^T$, $\omega_n \approx -0.346$]  % \omega_n = \frac{1}{2}\log \frac{1}{2}
  {
   \begin{tikzpicture}
    % Background layer for the order cone
    \pgfdeclarelayer{background}
    \pgfsetlayers{background,main}

    \begin{axis}[
      axis x line=middle,
      axis y line=middle,
      axis line style={<->,color=black, thick},
      xtick={-5,-4,...,5},
      ytick={0,5,...,20},
      xlabel={$\bar{\zeta}$},
      ylabel={$\disthyp{\omega}$},
      xlabel style={below right},
      ylabel style={below right},
      width=0.5\textwidth,
      xmin=-2,xmax=2,
      ymin=-2,ymax=20]

     % Draw the order cone
     \begin{pgfonlayer}{background}
      \fill [color=black!10!white] (0,-0.5) rectangle (2.2,0.5);
     \end{pgfonlayer}

     % Squared distance to the point a = (2,0.25)
     \addplot [very thick, mark=none,samples=30,domain=-2:2,variable=zeta] {0.5*(exp(zeta)-2)^2 + 0.5*(exp(-zeta)-0.25)^2};
    \end{axis}
  \end{tikzpicture}
  }
  \hspace{0.05\textwidth}
  \subfigure[Not convex: $a = (4,3)^T$, $\omega_n \approx 1.242$]  % \omega_n = \frac{1}{2}\log 12
  {
   \label{fig:da_hyp_2d_not_convex}
   \begin{tikzpicture}
    % Background layer for the order cone
    \pgfdeclarelayer{background}
    \pgfsetlayers{background,main}

    \begin{axis}[
      axis x line=middle,
      axis y line=middle,
      axis line style={<->,color=black, thick},
      xtick={-5,-4,...,5},
      ytick={0,5,...,30},
      xlabel={$\bar{\zeta}$},
      ylabel={$\disthyp{\omega}$},
      xlabel style={below right},
      ylabel style={below right},
      width=0.5\textwidth,
      xmin=-2,xmax=2,
      ymin=-2,ymax=20]

     % Draw the order cone
     \begin{pgfonlayer}{background}
      \fill [color=black!10!white] (0,-0.5) rectangle (2.2,0.5);
     \end{pgfonlayer}

     % Squared distance to the point a = (4,3)
     \addplot [very thick, mark=none,samples=30,domain=-2:2,variable=zeta] {0.5*(exp(zeta)-4)^2 + 0.5*(exp(-zeta)-3)^2};
    \end{axis}
   \end{tikzpicture}
  }
 \end{center}

 \caption{The squared distance $\disthyp{\omega}$ to $\sl(2)$, in hyperbolic coordinates.
  Left: From a point $a$ with $\product(a) < 1$, Right: From a point $a$ with $\product(a) > 1$.
  The order cone is the positive half-axis.}
 \label{fig:da_hyp_2d}
\end{figure}
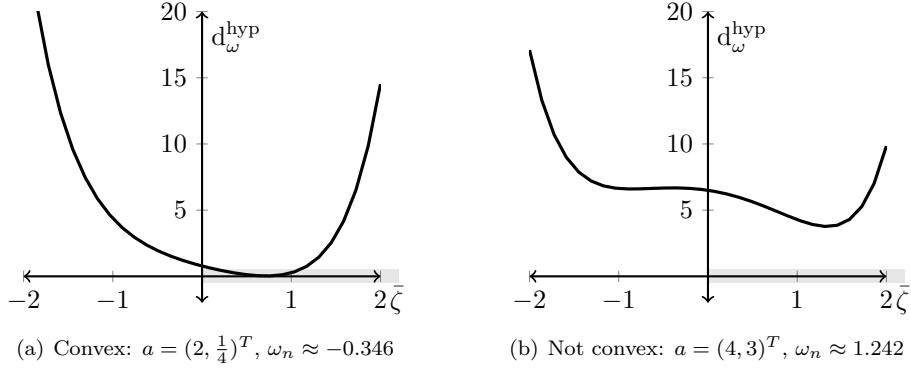

\begin{figure}
 \begin{center}
 \subfigure[$a = (6,5,4.5)^T$, $\omega_n \approx 1.635$]{%
   \includegraphics{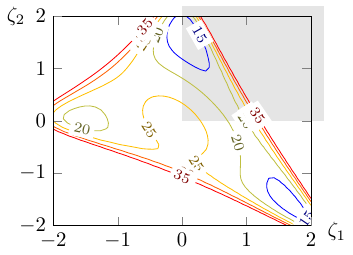}
  }
  \subfigure[$a = (5,1.2,1.1)^T$, $\omega_n \approx 0.629$]{%
    \includegraphics{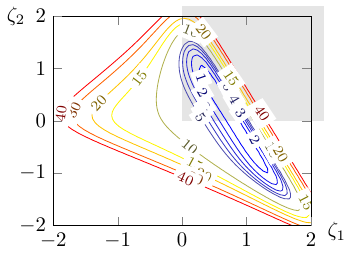}
  }
 \end{center}

 \caption{Level lines of $\disthyp{a}$ for $n=3$ for two points $a$ with $\product(a) > 1$
 (the grey square visualizes the order cone). The left picture shows a scenario
 with several local minimizers. The right picture illustrates that the energy landscape
 can exhibit long ``valleys'', which makes finding minimizers an algorithmic challenge.}
 \label{fig:da_hyp_3d}
\end{figure}

Finally, note that because of
\begin{equation*}
 \omega_n
 =
 \frac{1}{n} \sum_{i=1}^n \log a_i
 =
 \frac{1}{n} \log \product(a),
\end{equation*}
the term $\exp \omega_n$ in~\eqref{eq:disthyp_reduced} is
\begin{equation*}
 \exp \omega_n
 =
 \exp \Big(\frac{1}{n} \log \product(a)\Big)
 =
 \sqrt[n]{\product(a)},
\end{equation*}
which generalizes the geometric mean~\eqref{eq:2d_geometric_mean}
of the classical two-dimensional hyperbolic coordinates.
The important question whether $\product(a)$ is above or below~$1$ can therefore
be decided by looking at $\omega_n$ alone.

\section{Algorithms that compute the projection}
\label{sec:algorithms}

We now present four algorithms that compute the closest-point projection onto $\SL(n)$,
based on the different problem formulations encountered so far.
All algorithms operate in the space of diagonal matrices with nonnegative entries,
which we identify again with $\R^n_{\ge 0}$. To arrive at this representation,
each algorithm starts by computing the singular value decomposition of the
input matrix $A$. To save space we will omit this decomposition from all pseudocode
listings that appear in this chapter.

\subsection{Root-finding for the Lagrange multiplier}

\label{sec:direct_root_finding}

The first algorithm uses the solution path $\mathcal{P} : \big(-\infty,\frac{a_n^2}{2}\big] \to C_n$
introduced in Chapter~\ref{sec:uniqueness_in_the_order_cone}.
Recall that this is a continuous path with $\mathcal{P}(0) = a$ and
$\mathcal{P}(\tfrac{a_n^2}{2}) = \bar{a} \coloneqq (a_1,\dots,a_{n-1},0)$,
and such that if
\begin{equation}
\label{eq:root_finding}
 \product \big(\mathcal{P}(\lambda)\big)=1
\end{equation}
then the pair $(\mathcal{P}(\lambda),\lambda)$
solves the stationarity conditions~\eqref{eq:lagrange_vector_formulation}.

Since~\eqref{eq:root_finding} is a root-finding problem for a scalar, continuous
function, it can be solved using standard methods like bisection or regula falsi.

Suitable initial search intervals can be constructed based on knowing $\product(a)$.
If $\product(a) > 1$ then we know that there is (at least) one $\lambda$ that
solves~\eqref{eq:root_finding} in $\big[0,\frac{a_n^2}{2}\big]$, whereas
if $\product(a) < 1$, by Lemma~\ref{lem:solution_path_properties} the equation~\eqref{eq:root_finding}
has a solution in the set $(-\infty, 0]$, and this solution is unique. By the strict convexity shown
in Lemma~\ref{lemma:G+-is-convex} this is the unique global minimizer of the squared
Euclidean distance to the point $a$ in the positive orthant.

For $\product(a) < 1$ we can even show that the unique solution of~\eqref{eq:root_finding} is contained
in the interval $[-1,0]$.  Indeed, since we assume that $a_i \ge 0$ the stationarity
conditions imply that for the minimizer $p\in\sl(n)$
\begin{equation*}
  p_i + \frac{\lambda}{p_i} = a_i \ge 0,
\end{equation*}
which implies $\lambda \ge -p_i^2$ for all $i=1,\hdots,n$.  Since the $p_i$ cannot
all be larger than~$1$,
$\lambda \ge -1$ is a necessary condition for the minimizer.

As $\mathcal{P}$ is differentiable except at $\lambda = \frac{a_n^2}{4}$, a faster root-finding method
like regula falsi is a natural choice. However, while faster, in practical experiments
we found it to be less stable than the plain bisection method.

In summary, we arrive at the following algorithm:

\begin{algorithm}[H]
  \SetKwInOut{Input}{input}
  \SetKwInOut{Output}{output}
  \DontPrintSemicolon

  \Input {$a \in \R^n_{\ge 0}$}
  \BlankLine

  \eIf{$\product(a) < 1$}
  {Do (e.g.) bisection on $[-1,0]$ to determine the $\lambda$ with $\product \big(\mathcal{P}(\lambda)\big) = 1$ \;}
  {Do (e.g.) bisection on $[0,\frac{a^2_n}{2}]$ to determine a $\lambda$ with $\product \big(\mathcal{P}(\lambda)\big) = 1$ \;}

  \BlankLine
  \Output {$\mathcal{P}(\lambda)$}

  \caption{Root-finding for the Lagrange multiplier}
\end{algorithm}

\subsection{Composite step minimization}
\label{sec:composite_step}

The second algorithm is a conforming algorithm, i.e., the iterates $p^k$ it produces
all satisfy $p^k \in \sl(n)$. The algorithm has been inspired by
a work by \citeauthor{Elser2017} on matrix factorizations~\cite{Elser2017}.

Let $k \in \N$ be the iteration counter and $p^k \in \sl(n) \cap \R^n_{\ge 0}$.
Recall that $T_{p^k} \sl(n) \subset \R^n$ is the tangent plane of $\sl(n)$ at $p^k$.
At each iteration, the algorithm computes the line through $a$ that is normal
to $T_{p^k} \sl(n)$, and sets $p^{k+1}$ to be the unique intersection
of this line with $\sl(n) \cap \R^n_{\ge 0}$ (Figure~\ref{fig:composite_step_algorithm}).
More formally, we get the following algorithm:\\
\smallskip
\begin{algorithm}[H]
    \label{alg:composite_step}
    \SetKwInOut{Input}{input}
    \DontPrintSemicolon

    \Input {$a \in \R^n_{\ge 0}$, initial iterate $p^0 \in \sl(n)$}
    \BlankLine

    \ForEach{$k=0,1,2,\dots$}
    {
      \CommentSty{Compute normal vector $d^k$ of $T_{p^k}\textrm{sl}(n)$:} \;
      $d^k \leftarrow \nabla \product(p^k) = (p^k)^{-1}$ \;

      \CommentSty{Compute interval to search intersection in:} \;
      \eIf{$\product(a) < 1$}
      {
        $t_\textup{min} \leftarrow 0$  \;
        $t_\textup{max} \leftarrow 1 - \sqrt[n]{\product(a)}$  \;
      }
      {
        $t_\textup{min} \leftarrow \max_i (-a_i p^k_i)$ \;
        $t_\textup{max} \leftarrow 0$ \;
      }

      Do bisection on $[t_\textup{min}, t_\textup{max}]$ for unique zero $t^k$
      of $t \mapsto \product(a + td^k)-1$ \;
      \label{line:composite_step_bisection}

      $p^{k+1} \leftarrow a + t^k d^k$ \;
    }

    \caption{Composite-step minimization}
\end{algorithm}

The well-posedness of the intersection problem in Line~\ref{line:composite_step_bisection}
is justified by the following lemma.

\begin{lemma}
\label{lem:composite_step_step_size_nonpositive}
 For any $a \in \R_{\ge 0}^n$ and $p \in \sl(n) \cap \R^n_{\ge 0}$,
 there is a unique $t_*$ in the interval $[t_\text{min}, t_\text{max}]$
 such that $a + t_* p^{-1} \in \sl(n) \cap \R^n_{\ge 0}$,
 with $t_\textup{min} \colonequals \max_i (-a_i p_i)$
 and $t_\textup{max} \colonequals 1 - \sqrt[n]{\product(a)}$.
 This $t_*$ is in $[t_\textup{min}, 0]$ if $\product(a) \ge 1$, and in $[0,t_\textup{max}]$
 otherwise.
\end{lemma}
\begin{proof}
 Note first that $a + t p^{-1}$ is in $\R^n_{\ge 0}$ if and only if $t \ge t_\textup{min}$.
 At $t = t_\textup{min}$, we have $\product(a + tp^{-1}) = 0$, and
 for all $t > t_\textup{min}$, the expression $\product(a + tp^{-1})$ is a polynomial
 with positive real coefficients, and therefore $\product(a + tp^{-1}) \to \infty$
 for $t \to \infty$. Its first derivative is also a polynomial with positive coefficients.
 Therefore $t \mapsto \product(a + tp^{-1})$ is strictly monotone increasing.
 This proves that there is a unique $t_*$ with $\product(a+t_*p^{-1}) = 1$ in the interval $[t_\textup{min},\infty)$.

 Since $\product(a + tp^{-1}) = \product(a)$ at $t=0$, we must have $t_* \le 0$
 if $\product(a) \ge 1$ and $t_* \ge 0$ if $\product(a) \le 1$.
 To show that $t_* \le t_\textup{max}$ in the latter case, note that $t_* > 0$
 can only happen if $\product(a) \le 1$.  For $t \ge 0$ one can then invoke
 the Minkowski determinant theorem~\cite{marcus_minc:1992} to get
 \begin{equation*}
  1
  =
  \product(a + t_*d^k)^{\frac{1}{n}}
  \ge
  \product(a)^{\frac{1}{n}}
  +
  \product(t_*d^k)^{\frac{1}{n}}
  =
  \product(a)^{\frac{1}{n}}
  +
  t_*.
 \end{equation*}
 This implies $t_* \le 1 - \product(a)^{\frac{1}{n}} \equalscolon t_\textup{max}$.
\end{proof}

To avoid the costly $n$-th root it is possible to use the slightly weaker bound
$t_* \le 1$ instead.

\medskip

To investigate the properties of Algorithm~\ref{alg:composite_step} we rewrite it as a
composite-step method (see, e.g., \cite[Chapter~15.4]{conn_gould_toint:2000}).
Steps of such a method consist of two substeps: A ``tangent step'' that
decreases the objective functional, and a ``normal step'' that tries
to restore feasibility.  In our case we get:
\begin{description}
 \item[Tangent step] \label{item:tangent_step}
  Compute $p^{k+\frac{1}{2}}$, the unique minimizer of $\dist{a}$
  on the tangent space of $\sl(n)$ at $p^k$.

 \item[Normal step] \label{item:normal_step}
  Do a line search on the line from $a$ through $p^{k+\frac{1}{2}}$
  to find the unique point $p^{k+1} \in \sl(n) \cap \R^n_{\ge 0}$ on that line.
\end{description}

To see the equivalence of the two formulations, note that if $p^{k+\frac{1}{2}}$ really is
the minimizer of the squared Euclidean distance to $a$ on the tangent space $T_{p^k} \sl(n)$,
then $a-p^{k+\frac{1}{2}}$ must be orthogonal to that space.
Hence searching on the line that contains $a$ and $p^{k+\frac{1}{2}}$
is the same as searching in the normal direction $d^k \coloneqq \nabla \product(p^k)$
of $T_{p^k} \sl(n)$.

\begin{figure}
	\centering
	\begin{tikzpicture}

 % Background layer for the order cone
 \pgfdeclarelayer{background}
 \pgfsetlayers{background,main}

 \begin{axis}[
  axis x line=middle,
  axis y line=middle,
  axis line style={<->,color=black, thick},
  x label style={anchor=north},
  y label style={anchor=east},
  xmin=-0.5,xmax=4,
  ymin=-0.5,ymax=3,
  xtick={0,5,...,},
  ytick={0,5,...,},
  ]

  % Draw the order cone
  \begin{pgfonlayer}{background}
   \fill [color=black!10!white] (-0.2,-0.2) -- (3.8,-0.2) -- (3.8,3.0) -- (3,3) -- cycle;
  \end{pgfonlayer}

  % Danke GeoGebra für deinen tollen Tikz-Export :)
  \draw[blue,thick, smooth,samples=50,domain=0.4411569133490857:4.070203758747689] plot(\x,{1/\x});

  %\draw[red] (3.28,2.18)-- (1.27,0.79);
  \draw [dashed,domain=0.44:4.07] plot(\x,{(--1.58-0.62*\x)/1});
  \draw[red] (3.28,2.18)-- (2.1,0.27);
  \draw [dashed,domain=0.44:4.07] plot(\x,{(--0.9-0.2*\x)/1});
  \draw[red] (3.28,2.18)-- (2.9,0.31);

  \fill  (3.28,2.18) circle (1.5pt);
  \draw (3.3,2.21) node[right] {$a$};
  \fill  (1.27,0.79) circle (1.5pt);
  \draw (1.29,0.82) node[above] {$p^k$};
  \fill  (2.1,0.27) circle (1.5pt);
  \draw (2.12,0.35) node[below left] {$p^{k+\frac12}$};
  \fill  (2.21,0.45) circle (1.5pt);
  \draw (2.23,0.48) node[above] {$p^{k+1}$};
  \fill  (2.9,0.31) circle (1.5pt);
  \draw (3.15,0.35) node[below left] {$p^{k+\frac32}$};
  \fill  (2.91,0.34) circle (1.5pt);
  \draw (2.93,0.37) node[above right] {$p^{k+2}$};

  \draw[blue] (0.8,2) node{$\sl(n)$};

 \end{axis}

\end{tikzpicture}
	\caption{Two steps of the composite-step minimization algorithm.
	Tangent spaces are dashed, search lines are red.}
	\label{fig:composite_step_algorithm}
\end{figure}
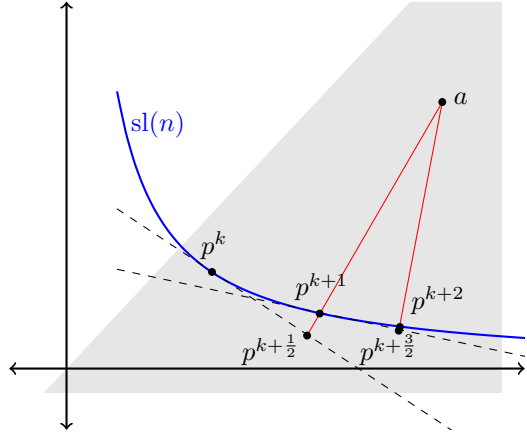

If $\product(a) \ge 1$, then any sequence of iterates started from within the order cone
stays within the order cone.

\begin{lemma}
\label{lem:cs_remains_in_cone}
  If $a\in C_n$ with $\product(a) \ge 1$, then $p^k \in C_n$ implies $p^{k+1} \in C_n$.
\end{lemma}
\begin{proof}
 By construction, $p^{k+1} = a + t^k d^k = a + t^k (p^k)^{-1}$,
 and by Lemma~\ref{lem:composite_step_step_size_nonpositive}, $t^k$ is nonpositive if $\product(a) \ge 1$.
 The components of $p^k$ are ordered and positive, therefore the components of
 $-(p^k)^{-1}$ are also ordered. The assertion follows because
 the sum of two elements in $C_n$ is also in $C_n$.
\end{proof}

The condition $\product(a) \ge 1$ is necessary for the assertion of Lemma~\ref{lem:cs_remains_in_cone}.
If $\product(a) < 1$, the iterates may leave the order cone.

We now prove that accumulation points of the algorithm are stationary points of the
squared distance $\dist{a}(p) \colonequals \frac{1}{2} \norm{p-a}^2$.
First, we claim that if $a \in \R^n_{\ge 0}$ then
the method is monotone, i.e., no iterate is ever farther away from $a$ than its predecessor.
This is easy to prove if $\product(a) \ge 1$.

\begin{lemma}[Energy decrease]
\label{lem:composite_step_energy_decrease}
 If $\product(a)\ge 1$ the sequence of points $(p^k)$ produced by the composite-step algorithm
 fulfills $\dist{a}(p^{k+1}) \le \dist{a}(p^k)$ for all $k=0,1,2,\dots$
\end{lemma}
\begin{proof}
We have $\dist{a}(p^{k+\frac12}) \le \dist{a}(p^k)$ by construction.
The strict convexity of $\sl^+(n) \cap \R^n_{\ge 0}$ (Lemma~\ref{lemma:G+-is-convex})
implies that $T_{p^k}\sl(n)\cap \sl(n) \cap \R^n_{\ge 0} = \{ p^k\}$, and hence
$p^{k+\frac{1}{2}} \notin \sl^+(n) \cap \R^n_{\ge 0}$ unless $p^{k+\frac{1}{2}} = p^k$.
Since $p^{k+1}$ is on the line through $p^{k+\frac12}$ and $a$, and between
these two points, we get $\dist{a}(p^{k+1}) \leq \dist{a}(p^{k+\frac{1}{2}})$.
\end{proof}

Numerical evidence suggests that monotonicity holds for $\product(a) < 1$ as well.
\begin{conjecture}
 \cref{lem:composite_step_energy_decrease} holds even if $\product(a) < 1$.
\end{conjecture}

The consequence of this monotonicity is that all iterates are contained in the set
\begin{equation*}
 \mathcal{L}_0
 \colonequals
 \Big\{ x \in \sl(n) \cap \R_{\ge 0}^n \; : \; \dist{a}(x) \le \dist{a}(p^0)\Big\}.
\end{equation*}

We now show that accumulation points of the joint sequence $(p^k, -t^k)$ of iterates
and negative step lengths are stationary points
of $\dist{a}$ on $\sl(n)$.  Note that Lemmas~\ref{lem:composite_step_step_size_nonpositive}
and~\ref{lem:composite_step_energy_decrease}
together with the coercivity of~$\dist{a}$ imply that there is at least
one accumulation point.
We stress that our proof only covers the case $\product(a) \ge 1$,
but that we observe convergence experimentally also for $0 < \product(a) < 1$.

\begin{theorem}\label{thm:convergence-monotone-descent}
 Let $\product (a) \ge 1$, and for each \(p^k\) generated by Algorithm~\ref{alg:composite_step},
 let \(t^k\in\R\) be the corresponding step length.  Then the sequence $(p^k, t^k)$
 fulfills
	\begin{equation*}
  \lim_{k\to\infty} \norm[\big]{ p^k - a + (- t^k) (p^k)^{-1} } = 0.
	\end{equation*}
 In particular, any convergent subsequence converges to a solution of the stationarity equations \eqref{eq:lagrange_vector_formulation}, with $-t^k$ converging to $\lambda$.
\end{theorem}

For the proof, we need more refined information about the function
\begin{equation*}
 E : \R^n \to \R,
 \qquad
 E(x) = \sum_{i=1}^n \exp x_i
\end{equation*}
introduced in Chapter~\ref{sec:logarithmic_distance}.
Strict convexity, coercivity, and the existence of a unique minimizer at $x=0$
have been shown in Lemma~\ref{lem:hyperbolic_cosine_is_convex}.
We can show even more when considering only the restriction of $E$
to a bounded domain.

\begin{lemma}
 \label{lem:sumexp}
 Let $B_\rho$ be the closed Euclidean unit ball of radius $\rho > 0$, and
 define
 \begin{equation}
 \label{eq:sumexp_min_eigenvalue}
  m_\textup{min} \coloneqq \min_{\substack{x \in B_\rho\\ i=1,\dots,n}} \exp x_i.
 \end{equation}
 Then
 \begin{equation*}
  \norm{y}^2 \le \frac{2}{m_\textup{min}} (E(y)-n).
 \end{equation*}
 for all $y \in B_\rho$.
\end{lemma}

\begin{proof}
 The Hesse matrix of $(E(\cdot) -n)$ at $x$ is
 $\operatorname{diag}\big(\exp x_1,\dots,\exp x_n\big)$.
 By construction, $m_\text{min} > 0$ bounds all eigenvalues from below on $B_\rho$.
 Hence, $(E(\cdot)-n)$ is strongly convex on $B_\rho$.
 As shown in~\cite[Chapter~9.1.2]{boyd_vandenberghe:2004},
 this strong convexity can be reformulated to
 \begin{equation*}
  E(y)-n \ge E(x) -n + \nabla (E(x)-n)^T(y-x) + \frac{m_\text{min}}{2}\norm{y-x}^2
  \qquad \forall x,y \in B_\rho.
 \end{equation*}
 In particular, for the global minimizer $x = 0$, we get
 \begin{equation*}
  E(y)-n \ge \underbrace{E(0)-n}_{=0} + \frac{m_\text{min}}{2}\norm{y}^2
  \qquad \forall y \in B_\rho.
 \end{equation*}
 This is the assertion.
\end{proof}

\begin{proof}[Proof of Theorem~\ref{thm:convergence-monotone-descent}]
 First, note that the sequence $(\dist{a}(p^k))$ is convergent,
 because it is monotone decreasing and bounded from below.
 This implies that for each $\varepsilon>0$ there is a $k_0\in\N$ with
 \begin{equation*}
  0 \le \dist{a}(p^k) - \dist{a}(p^{k+1}) < \frac{1}{2}\varepsilon^2
  \qquad \forall k>k_0.
 \end{equation*}
 Since $\product(a) \ge 1$ we have $\dist{a}(p^{k+1}) \leq \dist{a}(p^{k+\frac12}) \leq \dist{a}(p^k)$, and hence also
 \begin{equation}\label{eq:q^i-minus-p^i+1}
  0 \le \dist{a}(p^{k+\frac12}) - \dist{a}(p^{k+1}) < \frac{1}{2}\varepsilon^2
  \qquad \forall k>k_0.
 \end{equation}
 Monotonicity of the square function implies
 \begin{equation*}
  0 \le \norm{p^{k+\frac12} - a} - \norm{p^{k+1} - a} < \varepsilon
  \qquad \forall k>k_0,
 \end{equation*}
 and since $p^{k+\frac{1}{2}}$ and $p^{k+1}$ are on the same ray from the point $a$
 we get
 \begin{equation*}
  \norm{p^{k+\frac{1}{2}} - p^{k+1}} < \varepsilon
  \qquad
  \text{for all $k > k_0$.}
 \end{equation*}
 The rest of the proof now shows that this can be used to bound
 $\norm{p^k - p^{k+1}}$, more specifically that there is a constant $K > 0$ such that
 for all $p^k \in \mathcal{L}_0$
 \begin{equation}
 \label{eq:sqrt_bound}
  \norm{p^k - p^{k+1}}
  \le
  K
  \sqrt{ \norm{p^{k+\frac12} - p^{k+1}} },
 \end{equation}
 which implies that
 \begin{equation*}
  \norm{p^k - p^{k+1}} \to 0.
 \end{equation*}
 This yields the assertion, because
 inserting the definition of $p^{k+1} \colonequals a + t^k (p^k)^{-1}$ yields
 \begin{equation*}
  \norm[\big]{p^k - a - t^k (p^k)^{-1}} \to 0.
 \end{equation*}

 The proof of~\eqref{eq:sqrt_bound} uses the Lie exponential $\exp$~\eqref{eq:exponential_map},
 which maps $T_\mathbbm{1} \sl(n)$ to $\sl(n) \cap \R^n_{\ge 0}$.
 The main idea of the proof is that since $\exp$ is an isomorphism,
 there is a $v \in T_\mathbbm{1} \sl(n)$ such that
 \begin{equation*}
  p^{k+1}
  =
  p^k \exp v,
 \end{equation*}
 and we can bound the distance from $p^{k+\frac 12}$ to $p^{k+1}$ from below
 in terms of the tangent vector $v$.  To do so, note that by construction,
 the point $p^{k+\frac 12}$ is the orthogonal projection of $p^{k+1}$ onto $T_{p^k} \sl(n)$.
 Hence, the distance between the two points is equal to the distance between $p^{k+1}$
 and the space $T_{p^k} \sl(n)$.  Using that $\frac{(p^k)^{-1}}{\norm{(p^k)^{-1}}}$ is a unit normal
 of $T_{p^k} \sl(n)$, that distance is
 \begin{align*}
  \norm[\big]{p^{k+\frac12} - p^{k+1}}
  & =
  \operatorname{dist} \big(p^{k+1}, T_{p^k}\sl(n) \big)
  \\
  & =
  \bigg \langle p^k \exp v - p^k, \frac{(p^k)^{-1}}{\norm{(p^k)^{-1}}} \bigg \rangle
  \\
  & =
  \bigg \langle \exp v - \mathbbm{1}, \frac{\mathbbm{1}}{\norm{(p^k)^{-1}}} \bigg \rangle
  \\
  & =
  \frac{1}{\norm{(p^k)^{-1}}} \Big( \sum_{i=1}^n \exp v_i - n \Big)
  \\
  & =
  \frac{1}{\norm{(p^k)^{-1}}} \big(E(v)-n\big).
 \end{align*}

 As the logarithm (the inverse exponential map) is continuous and $\mathcal{L}_0$ is compact,
 we can infer that there is a constant $\rho > 0$ such that
 $\norm{v} = \norm{\log (p^{k+1}(p^k)^{-1})} \le \rho$
 for all $p^k, p^{k+1} \in \mathcal{L}_0$.
 This allows to use the strong convexity of $E$ on bounded sets to get a lower bound
 of $\norm{p^{k+\frac12} - p^{k+1}}$ in terms of $\norm{v}$. Let $m_\text{min}$
 be the constant defined in~\eqref{eq:sumexp_min_eigenvalue} for the ball $B_\rho$.
 Then Lemma~\ref{lem:sumexp} yields
 \begin{align}
  \label{eq:vertical_bound}
  \norm[\big]{p^{k+\frac12} - p^{k+1}}
  \ge
  \frac{1}{\norm{(p^k)^{-1}}} \frac{m_\text{min}}{2} \norm{v}^2.
 \end{align}

 We can now bound the desired quantity $\norm[\big]{p^k-p^{k+1}}$.
 Let $L$ be the Lipschitz constant of $\exp$ as a map $\R^n \to \R^n$
 for the bounded set $B_\rho$, and let $p^k_\text{max} \colonequals \max_{i=1,\dots,n} p_i^k$.
 Then
 \begin{align*}
  \norm[\big]{p^k - p^{k+1}}
  & =
  \norm[\big]{p^k \big( \exp 0 - \exp v \big)}
  \\
  & \le
  p^k_\text{max} \norm[\big]{\exp 0 - \exp v}
  \\
  & \le
  p^k_\text{max} L \norm{0 - v}.
 \end{align*}
 Finally, combining this with~\eqref{eq:vertical_bound} we get
 \begin{align*}
  \norm{p^k - p^{k+1}}
  & \le
  p^k_\text{max} L \norm{v}
  \\
  & \le
  p^k_\text{max} L\sqrt{\frac{2 \norm{(p^k)^{-1}}}{m_\text{min}} \norm{p^{k+\frac12} - p^{k+1}}}.
 \end{align*}
 Since $p^k$ is contained in the bounded set $\mathcal{L}_0$ for all $k$,
 both $p^k_\text{max}$ and $\norm{(p^k)^{-1}}$ are bounded independently of $k$.
 Hence, there is a constant $K$ such that
 \begin{equation*}
  \norm{p^k - p^{k+1}}
  =
  K \sqrt{\norm{p^{k+\frac12}-p^{k+1}}},
 \end{equation*}
 and the assertion is shown.
\end{proof}

\begin{remark}
 The composite-step algorithm can also be formulated in the space of matrices.
 Indeed, given  a $P^k \in \SL(n)$, a normal vector to the tangent space there is,
 by Jacobi's formula,
 \begin{equation*}
  \nabla \det(P^k)
  =
  \det P^k (P^k)^{-T}
  =
  (P^k)^{-T}.
 \end{equation*}
 The update step of the algorithm is then a line search on the ray from $A$ in the direction of
 $(P^k)^{-T}$, and
 \begin{equation*}
  P^{k+1}
  =
  A + t^k(P^k)^{-T}.
 \end{equation*}

 Interestingly, this matrix form of the algorithm produces exactly the same iterates as the one in $\R^n$,
 modulo multiplication with orthogonal matrices.  For an induction proof, assume that
 $A$ and $P^k$ can be diagonalized by the same pair of orthogonal matrices
 $A = U \Sigma_A V^T$ and $P^k = U \Sigma_{P^k} V^T$, with $\Sigma_A$ and $\Sigma_{P_k}$ diagonal.
 Then $(P^k)^{-T} = U \Sigma_{P^k}^{-1} V^T$,
 and therefore $P^{k+1}$ is also diagonalized by the pair of matrices $U,V^T$.

 Using the algorithm in matrix space avoids the costly initial singular value decomposition.
 However, the line search for the step size $t^k$ then involves computing determinants of
 $n \times n$-matrices, which has similar cost. Therefore, the matrix form
 of the algorithm remains an elegant curiosity with little practical value.
\end{remark}

\subsection{Newton method with eliminated constraints}
\label{sec:newton_without_constraints}

In Chapter~\ref{sec:global_minimizer} we have introduced $n$-dimensional hyperbolic
coordinates.  They allowed to formulate minimizing $\dist{a}$ in $\sl(n) \cap \R^n_{\ge 0}$
without any constraint at all.
We even located the global minimizer in the positive orthant of these coordinates.
This suggests to use a standard descent method for unconstrained
smooth functionals such as the Newton method, with a
suitable line search criterion.
Unfortunately, plots of the energy landscape in hyperbolic coordinates
(such as Figure~\ref{fig:da_hyp_3d}) indicate that the problem may become
very ill-conditioned. The practical effects of this are studied in
Chapter~\ref{sec:experiments} below.

In pseudocode, a Newton algorithm for minimizing the squared distance to $a$
in hyperbolic coordinates has the following general form:

\begin{algorithm}[H]
    \label{alg:conforming_descent}
    \caption{Newton method with eliminated constraints}

    \SetKwInOut{Input}{input}
    \DontPrintSemicolon

    \Input {$a \in \R^n_{\ge 0}$, initial iterate $\bar{\zeta}^0 \in \R^{n-1}$}
    \BlankLine

    \ForEach{$k=0,1,2,\dots$}
    {
      \CommentSty{Compute Hesse matrix:} \;
      $H^k \leftarrow \nabla^2 \disthyp{a}(\bar{\zeta}^k)$ \;
      \If{$H^k$ not positive definite}
      {
        Modify $H^k$ \;
        \label{algline:newton_hesse_modification}
      }
      \CommentSty{Compute Newton direction and step length:} \;
      $c^k \leftarrow - (H^k)^{-1} \nabla \disthyp{a}(\bar{\zeta}^k)$ \;
      \label{algline:newton_correction}
      $\eta^k \leftarrow \text{step length}$ \;
      \CommentSty{Apply update:} \;
      $\bar{\zeta}^{k+1} \leftarrow \bar{\zeta}^k + \eta^k c^k$ \;
    }
    \CommentSty{Transform last iterate to Euclidean coordinates:} \;
    $p \leftarrow \exp (\bar{B}\bar{\zeta})$ \;
\end{algorithm}

Local quadratic convergence of this algorithm to a local minimizer
is a classic result and shown, e.g., in~\cite[Thm.\,3.5]{nocedal_wright:2006}.
This result assumes that the Hesse matrix $\nabla^2 \disthyp{a}$ is positive definite
in a neighborhood of the minimizer, and therefore does not require
the Hesse matrix modification step in Line~\ref{algline:newton_hesse_modification}.

Unfortunately, $\disthyp{a}$ is not strictly convex away from minimizers,
and therefore indefinite or singular Hesse matrices $\nabla^2 \disthyp{a}$
cannot be ruled out.
To see how the Hesse matrix can become singular, we note that at an iterate $\bar{\zeta}^k$
it is
\begin{equation*}
 \nabla^2 \disthyp{a}(\bar{\zeta}^k)
 =
 \bar{B}^T D^\text{hyp}(\bar{\zeta}^k) \bar{B},
\end{equation*}
where $\bar{B} \in \R^{n \times (n-1)}$ is the matrix formed by
the left $n-1$ columns of the transformation matrix $B$
defined in~\eqref{eq:hyperbolic_transformation_matrix}, and
\begin{equation}
\label{eq:def_of_d_hyp}
 D^\text{hyp}(\bar{\zeta}^k)
 \colonequals
 \operatorname{diag}_{i=1,\dots,n} \Big(\exp (\bar{b}_{i*} \bar{\zeta}^k) \big(2 \exp (\bar{b}_{i*} \bar{\zeta}^k) -a_i\big) \Big).
\end{equation}
The Hesse matrix
will turn singular if at least two entries of the diagonal factor vanish,
i.e., if
\begin{equation}
\label{eq:diagonal_zero}
 2 \exp (\bar{b}_{i*} \bar{\zeta}^k)
 =
 a_i
\end{equation}
for at least two indices.%
\footnote{
Note that in Cartesian coordinates the condition is simply $2p^k_i = a_i$.
}
(One index is not sufficient because the range
of $\bar{B}$ is the orthogonal complement of $\mathbbm{1}$, and therefore
cannot contain elements of the kernel of $D^\text{hyp}$ if only one diagonal
entry is zero.)
However, the Hesse matrix may be
singular even if $D^\text{hyp}(\bar{\zeta}^k)$ is regular! To see this
note that $\operatorname{ker} \bar{B}^T = \mathbbm{1}$, and that (as just mentioned)
the range of $\bar{B}$ is $\mathbbm{1}^\perp$, the orthogonal
complement of $\mathbbm{1}$. Hence the Hesse matrix is singular
if the range of the diagonal factor $D^\text{hyp}$ on $\mathbbm{1}^\perp$
contains a multiple of $\mathbbm{1}$, which happens if
\begin{equation}
\label{eq:trace_condition_hyperbolic}
 \Big\langle \mathbbm{1} \big(D^\text{hyp}(\bar{\zeta}^k)\big)^{-1}, \mathbbm{1} \Big\rangle
 =
 \sum_{i=1}^n \frac{1}{\exp (\bar{b}_{i*} \bar{\zeta}^k) (2 \exp (\bar{b}_{i*} \bar{\zeta}^k) -a_i)}
 =
 0.
\end{equation}
Possible modifications of the Hesse matrix to make it positive definite
are discussed in~\cite[Chapter\,3.4]{nocedal_wright:2006}.

\subsection{Newton method with explicit linear constraint}
\label{sec:newton_with_constraints}

The Newton matrices $\nabla^2 \disthyp{a}$ of
the previous section are dense, and therefore computing the Newton corrections $c^k$
will be expensive unless $n$ is small.
We can, however, exploit the additive structure of the original energy functional
to speed up the algorithm.
For this we go back to the formulation in the logarithmic coordinates
of Chapter~\ref{sec:logarithmic_coordinates}, where the problem was to find
a $\xi \in \R^n$ such that
\begin{align*}
 \min_{\xi \in \R^n} & \distlog{a}(\xi)
 \qquad \qquad
 \distlog{a}(\xi) \colonequals \frac{1}{2} \sum_{i=1}^n (\exp \xi_i - a_i)^2\\
 \xi_1 + \xi_2 + & \dots + \xi_n = 0.
\end{align*}
Given an admissible iterate $\xi^k \in \R^n$, i.e., $\sum_{i=1}^n \xi_i^k = 0$,
the constrained Newton method
computes a correction $c^k \in \R^n$ such that~\cite{boyd_vandenberghe:2004}
\begin{equation*}
 c^k
 =
 \argmin_{c_1 + \dots + c_n = 0}
    c^T \nabla \distlog{a}(\xi^k) + \frac{1}{2} c^T \nabla^2 \distlog{a}(\xi^k) c.
\end{equation*}
The first-order optimality condition of this is
\begin{equation}
\label{eq:constrained_newton_system}
 \begin{pmatrix}
  & & & 1 \\
  & \nabla^2 \distlog{a}(\xi^k) & & \vdots \\
  & & & 1 \\
  1 & \cdots & 1 & 0
 \end{pmatrix}
 \begin{pmatrix}
  c^k \\ w^k
 \end{pmatrix}
 =
 \begin{pmatrix}
  - \nabla \distlog{a}(\xi^k) \\
  0
 \end{pmatrix},
\end{equation}
where $w^k \in \R$ is the Lagrange multiplier of the constraint.
Since
\begin{equation*}
 \nabla^2 \distlog{a}(\xi)
 =
 \operatorname{diag}_{i=1,\dots,n} \big(\exp \xi_i (2 \exp \xi_i -a_i) \big)
\end{equation*}
is diagonal, the matrix in~\eqref{eq:constrained_newton_system} is an arrowhead matrix.
Such a matrix can be inverted in linear time, as can be verified
by direct computation.

\begin{lemma}[Inverse of arrowhead matrix]
\label{lem:arrowhead_inversion}
 If $M = \begin{psmallmatrix} D & z \\ z^T & 0 \end{psmallmatrix}$ with
 diagonal, invertible $D$ and $z = (1,\dots,1)^T$, then
 \begin{equation*}
  M^{-1}
  =
  \begin{pmatrix}
   D^{-1} & 0 \\
   0 & 0
  \end{pmatrix}
  +
  \rho u u^T,
 \end{equation*}
 where
 \begin{equation*}
  u = \begin{pmatrix} \operatorname{diag}(D^{-1}) \\ -1 \end{pmatrix}
  \qquad \text{and} \qquad
  \rho = - \operatorname{tr}(D^{-1})^{-1}.
 \end{equation*}
\end{lemma}

We arrive at an algorithm that looks essentially like Algorithm~\ref{alg:conforming_descent},
but with the inverse of the Newton matrix replaced by \cref{lem:arrowhead_inversion}.
It is shown in~\cite[Chapter~10.2.3]{boyd_vandenberghe:2004}
that both methods produce the same sequence of iterates,
modulo the transformation from logarithmic to hyperbolic coordinates.
Therefore, the same convergence guarantees hold.

Just as in hyperbolic coordinates, the matrices of the Newton method
with explicit constraint may not be invertible, but it is easier here
to see when this happens. From \cref{lem:arrowhead_inversion}
we can see that the bordered Hesse matrix of~\eqref{eq:constrained_newton_system}
is singular if and only if $\nabla^2 \distlog{a}(\xi^k)$ is singular,
of if $\operatorname{tr}\big(\nabla^2 \distlog{a}(\xi^k)^{-1} \big) = 0$.
In components this means that
\begin{equation}
\label{eq:invertibility_logarithmic}
 2 \exp \xi_i^k = a_i
\end{equation}
for at least one $i \in \{1,\dots,n\}$, or
\begin{equation}
\label{eq:trace_condition_hyperbolic_2}
 \sum_{i=1}^n \frac{1}{\exp \xi^k_i (2 \exp \xi^k_i - a_i)}
 =
 0.
\end{equation}
The latter is exactly condition~\eqref{eq:trace_condition_hyperbolic}
from \cref{sec:newton_without_constraints}.
Condition~\eqref{eq:invertibility_logarithmic} is the same as~\eqref{eq:diagonal_zero},
but now even a single index~$i$ with $2\exp \xi_i^k = a_i$ will
make the Newton matrix singular.

Modifications of the Hesse matrix can be handled by Lemma~\ref{lem:arrowhead_inversion}
as long as they change the diagonal of the Hesse matrix only.

\section{Numerical experiments}
\label{sec:experiments}

We compare the performance of the algorithms numerically,
using implementations in the \textsc{Julia} language. We test with random matrices
of sizes $n = 2,3,4,8,16,32$, and $64$.

\subsection{Generation of test matrices}

For each matrix size we create four sets of random test matrices.
As simply picking matrix entries randomly quickly leads to matrices
with astronomically large determinants, we construct the matrices
in logarithmic space. For a given $\epsilon > 0$,
we first construct a matrix $T$ with uniformly distributed entries
from the interval $[-(\log \epsilon) / n, (\log \epsilon)/ n]$.
The trace of that matrix is between $- \log \epsilon$ and $\log \epsilon$,
and it follows the Bates distribution~\cite{johnson_kotz_balakrishnan:1995}.
Since the variance of this distribution decreases with increasing~$n$,
we additionally multiply the entries with the factor $\sqrt{n}$, which neutralizes this effect.
Using then $\det \exp T = \exp \operatorname{tr} T$, the matrix
\(A \colonequals \exp(T)\) is a suitable test matrix with determinant
in the interval $[\epsilon^{-\sqrt{n}}, \epsilon^{\sqrt{n}}]$.

For the tests here we pick the value $\epsilon = 100$ and construct four sets of 1000 matrices each, namely $\mathcal{A}^{\ge 1}$, $\mathcal{A}^{< 1}$, $\mathcal A^0$ and $\mathcal A^{\overline{C}}$.
The sets $\mathcal{A}^{\ge 1}$ and $\mathcal{A}^{<1}$ contain only matrices $A$ with $\det A \ge 1$ and $\det A <1$, respectively.
The other sets are constructed by modifying the singular values of each matrix $A$,
to test the algorithms also in extreme situations.  We assume that the singular values
are sorted in nonincreasing order.
\begin{enumerate}
  \item $\mathcal A^0$ (singular matrices): Here we replace the last (i.e., smallest)
  \(\lceil \frac{n}{3} \rceil\)
  singular values of each random matrix~$A$ by~$0$.

  \item $\mathcal A^{\overline{C}}$ (boundary of the order cone):
  For each matrix \(A\) we randomly select \(\lfloor \frac{n}{3} \rfloor\) indices between 1 and \(n-1\).
  Then, for each maximal consecutive subset of these indices,
  we take the geometric mean of the corresponding singular values
  and the next smaller singular value not in the subset. The geometric mean
  then replaces all values it was computed from in the list of singular values.
  This moves $A$ onto the boundary of the order cone without changing its determinant.
\end{enumerate}

Note that for \(n=2\) the matrices in \(\mathcal A^0\) contain exactly one zero singular value.
Also for \(n=2\), the matrices in \(\mathcal A^{\overline C}\) are intentionally \emph{not}
moved to the boundary of \(C_2\), since otherwise \((1,1)^T\) would be a stationary point for all of them.
The distribution of the determinants of the matrices in $\mathcal A^{\overline{C}}$ is shown in \cref{fig:histogram_determinants}.

\begin{figure}
  \begin{center}
      \includegraphics[width=0.8\textwidth]{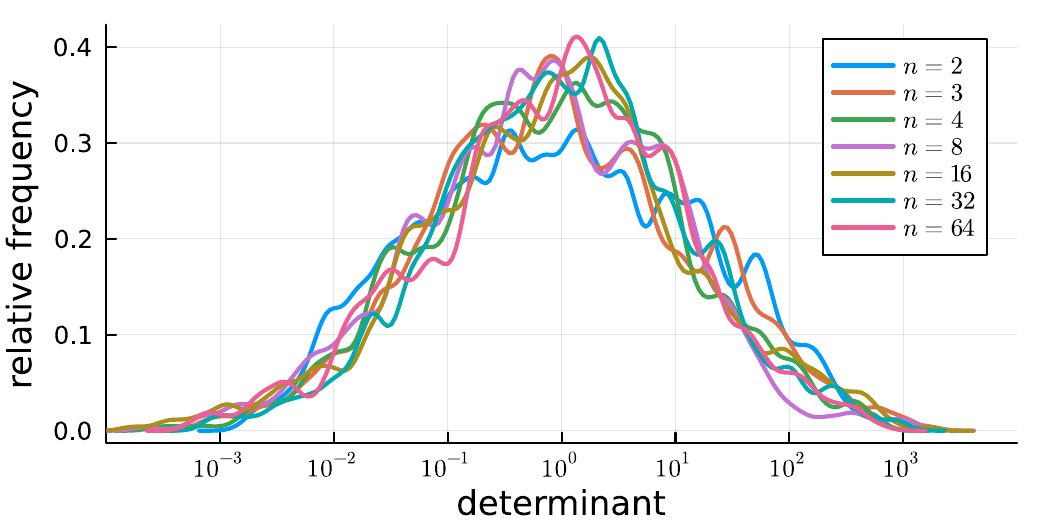}
  \end{center}
 \caption{Smoothed distribution of the determinants of the test matrices}
 \label{fig:histogram_determinants}
\end{figure}

\subsection{Test setup}

We now test all four algorithms on each of the four matrix test sets.
Plain bisection is used for all situations where a zero of a function has to be found;
this is slower but more stable than alternatives like regula falsi.
For both Newton methods a constant step length
equal to~$1$ is used, and the Hesse matrix is never modified.
This leads to convergence in almost all cases. Singular Hesse matrices appear in about 3\%
% um genau zu sein: für alle 28000 Matrizen (7 dims * 4 modi * 1000 each) hatten wir 849 failures: 28849/28000 = 1,030321 ~ 3%
 of all test cases,
and we simply remove those cases from our test set.

We let all iterations run until a given precision is reached, but by necessity
the exact termination criteria differ from algorithm to algorithm.
The root-finding bisection method stops when the current interval
becomes smaller than \(10^{-8}\), or when \(\mathcal{P}(\lambda^{k+1})\),
the point on the solution path at the next search interval midpoint,
has a determinant that differs from \(1\) by a value below  \(10^{-15}\).
The composite step method stops once the relative difference between the current and the previous iterate \(\frac{\norm{p^k - p^{k-1}}_\infty}{\norm{p^k}_\infty}\) is below \(10^{-8}\).
The Newton method without constraint and the one with
explicit constraint monitor the relative correction sizes
\( \frac{\norm{c^k}_\infty}{\norm{\xi^k}_\infty}\)
and \( \frac{\norm{c^k}_\infty}{\norm{\bar{\zeta}^k}_\infty}\), respectively,
and stop once these quantities drop below $10^{-8}$.
To avoid very long-running iterations, we let all methods report a failure if 200 iterations are exceeded. This happens only a few times for the Newton methods.

All methods except for the root-finding bisection require an initial iterate on $\sl(n)$.
Constructing good initial iterates is a surprisingly difficult issue.
The following approach seems to be reasonable.
It combines the radial scaling introduced
in Remark~\ref{rem:scaling} with an additional correction step
within $\sl(n)$. This additional step is constructed such that $p^0$
and $a$ coincide in the largest component, which is the first one
as we assume $a$ to be in the order cone.
If $a$ is below the tangent space at $\mathbbm{1}$, it is first projected
onto that space. By this, the projected point has the largest component
equal to or larger than~1, and the subsequent scaling along $\sl(n)$
will not leave the order cone.
Figure~\ref{fig:initial-value} illustrates the procedure for \(n=2\).

\begin{algorithm}[H]
  \label{alg:initial_value}
  \caption{Constructing an initial iterate}

  \SetKwInOut{Input}{input}
  \DontPrintSemicolon

  \Input {$a \in \R^n_{\ge 0}\cap C_n$}
  \BlankLine

  \If(\tcp*[f]{If $a$ is below $T_\mathbbm{1} \textrm{sl}(n)$})
  {$\langle \mathbbm{1}, a - \mathbbm{1} \rangle < 0$}
  {
    \CommentSty{Project onto \( T_\mathbbm{1} \textrm{sl}(n)\):} \;
    \(a \leftarrow \mathbbm{1} + a - \frac{1}{n}\langle \mathbbm{1},a \rangle\)
  }

   \CommentSty{Scale radially onto \(\textrm{sl}(n)\):} \;
   \( p^0 \leftarrow \product(a + 10^{-15}\mathbbm{1})^{-\frac{1}{n}} (a + 10^{-15}\mathbbm{1}) \)
   \tcp*{$10^{-15}\mathbbm{1}$ avoids division by $0$}

  \If{\(p^0_1 > 1\label{line:exclude-p-1}\)}
  {
     \CommentSty{Geodesic scaling in $\textrm{sl}(n)$:} \;
    \(p^0 \leftarrow \exp (\gamma \log p^0)\) with \(\gamma \colonequals \frac{\log a_1}{\log p^0_1 }\)
    \tcp*{Choice of $\gamma$ ensures $p_1^0 = a_1$}
  }
\end{algorithm}

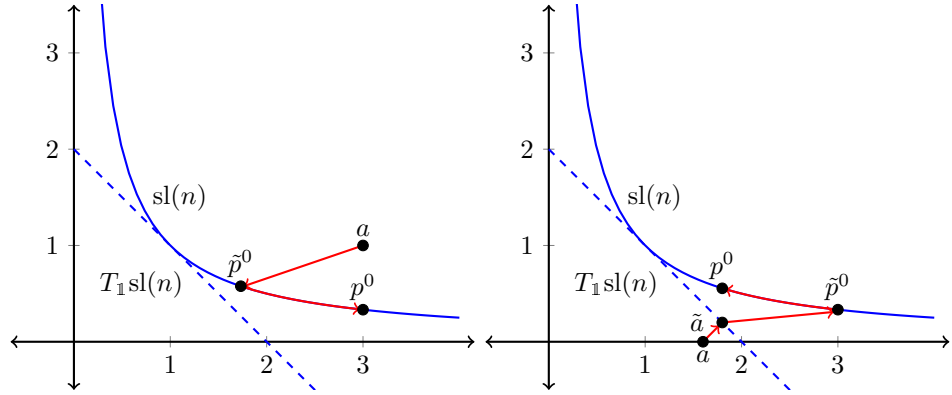
\begin{figure}
  \centering
  \begin{tikzpicture}

    \begin{axis}[
      axis x line=middle,    % put the x axis in the middle
      axis y line=middle,    % put the y axis in the middle
      axis line style={<->,color=black, thick}, % arrows on the axis
      xmin=0.5,xmax=3,
      ymin=-0.5,ymax=3.5,
      xtick={0,1,...,3},
      axis equal,
      scale=0.9
      ]

      % sl(2)
      \addplot [domain=0:4,samples=50,thick,color=blue]{1/x};
      \node at (1.1,1.5) {$\sl(n)$};

      % T_1 sl(2)
      \addplot [domain=0:4,samples=50,thick,color=blue, dashed]{2-x};
      \node at (0.7,0.6) {$T_\mathbbm{1} \sl(n)$};

      \draw[->,thick,red] (3,1) --  (1.762,0.577);
      \draw[->,thick,red] plot[domain=1.732:2.97, samples=10] ({\x}, {1/\x});
      \draw[fill] (3,1) circle (2pt) node[above]{\(a\)};
      \draw[fill] (1.732,0.577) circle (2pt) node[above]{\(\tilde p^0\)};
      \draw[fill] (3,0.333) circle (2pt) node[above]{\(p^0\)};

    \end{axis}
  \end{tikzpicture}
    \begin{tikzpicture}

    \begin{axis}[
      axis x line=middle,    % put the x axis in the middle
      axis y line=middle,    % put the y axis in the middle
      axis line style={<->,color=black, thick}, % arrows on the axis
      xmin=0.5,xmax=3,
      ymin=-0.5,ymax=3.5,
      xtick={0,1,...,3},
      axis equal,
      scale=0.9
      ]

      % sl(2)
      \addplot [domain=0:4,samples=50,thick,color=blue]{1/x};
      \node at (1.1,1.5) {$\sl(n)$};

      % T_1 sl(2)
      \addplot [domain=0:4,samples=50,thick,color=blue, dashed]{2-x};
      \node at (0.7,0.6) {$T_\mathbbm{1} \sl(n)$};

      \draw[->,thick,red]  (1.6, 0.0) --  (1.77, 0.17);
      \draw[->,thick,red]  (1.8, 0.2) --  (2.96, 0.31);
      \draw[->,thick,red] plot[domain=3:1.83, samples=10] ({\x}, {1/\x});
      \draw[fill] (1.6,0.0) circle (2pt) node[below]{\(a\)};
      \draw[fill] (1.8,0.2) circle (2pt) node[left]{\(\tilde a\ \)};
      \draw[fill] (3,0.3333) circle (2pt) node[above]{\(\tilde p^0\)};
      \draw[fill] (1.8,0.555) circle (2pt) node[above]{\(p^0\)};

    \end{axis}
  \end{tikzpicture}
  \caption{Constructing initial iterates $p^0$ for \(a=(2,1)^T\) and \(a=(1.6, 0)^T\).
  The symbols with a tilde represent the intermediate steps
  of Algorithm~\ref{alg:initial_value}.}
   \label{fig:initial-value}
\end{figure}

We construct the initial iterate for the two Newton methods
by transforming \(p^0\) to hyperbolic coordinates and logarithmic coordinates,
respectively.

\subsection{Results}

\begin{figure}
 \begin{center}
 \subfigure[$\mathcal{A}^{<1}$: Convex projection]{
  \includegraphics[width=0.4\textwidth]{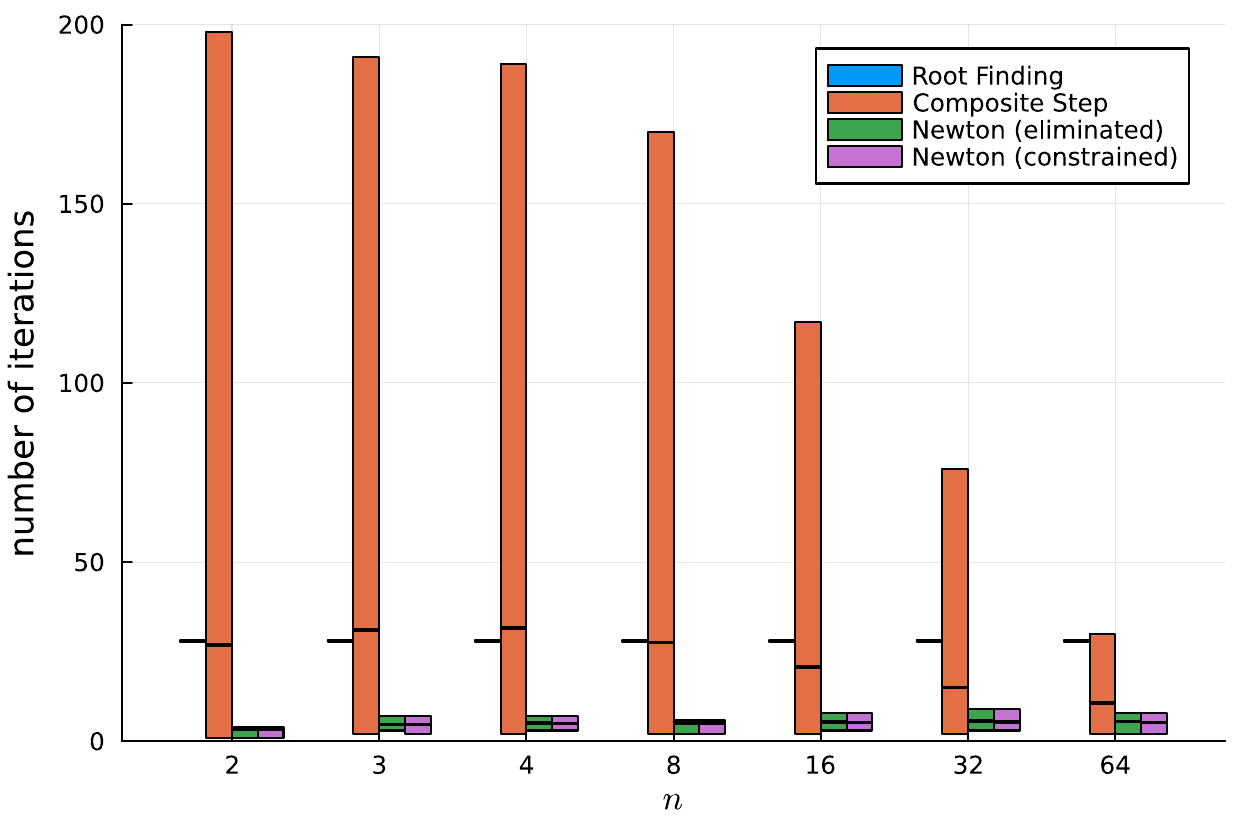}
 }
 \hspace{0.05\textwidth}
 \subfigure[$\mathcal{A}^{\ge 1}$: Non-convex projection]{
  \includegraphics[width=0.4\textwidth]{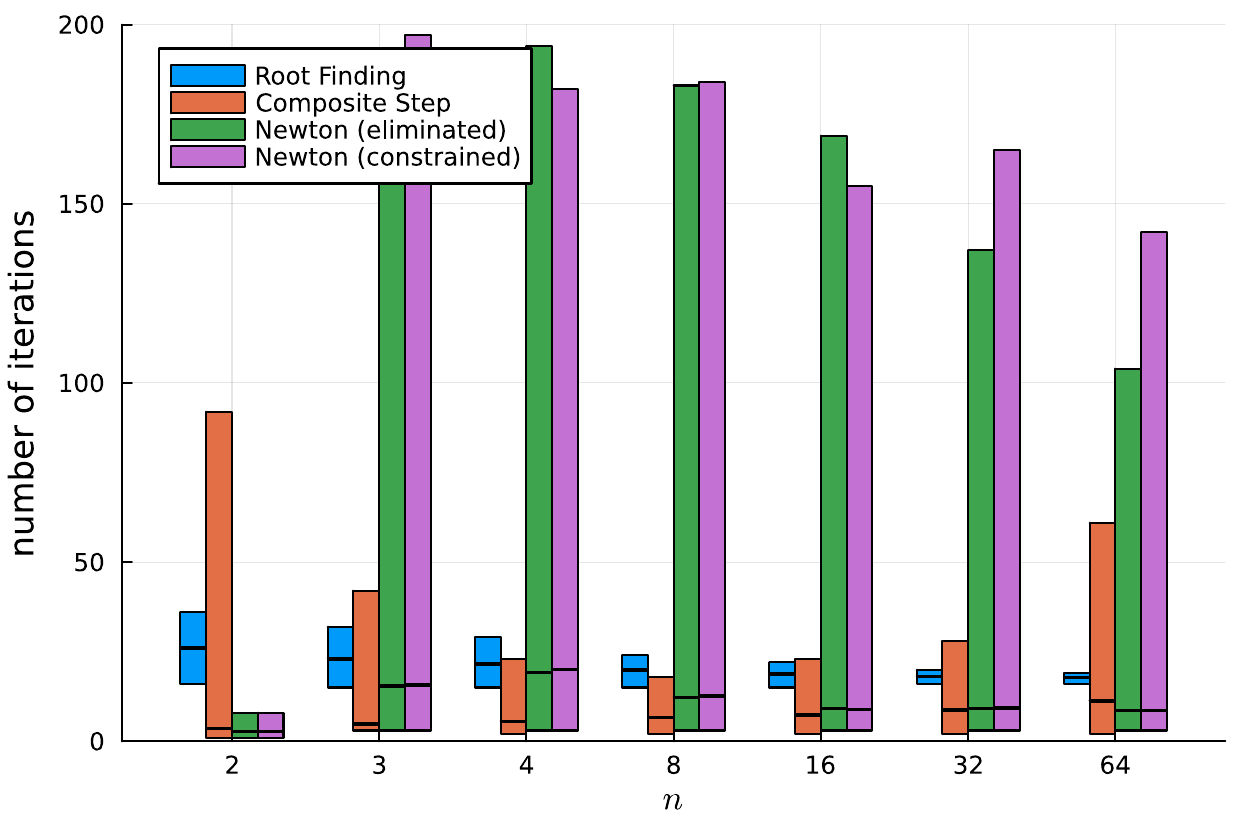}
 }

 \subfigure[$\mathcal{A}^0$: Singular matrices]{
  \includegraphics[width=0.4\textwidth]{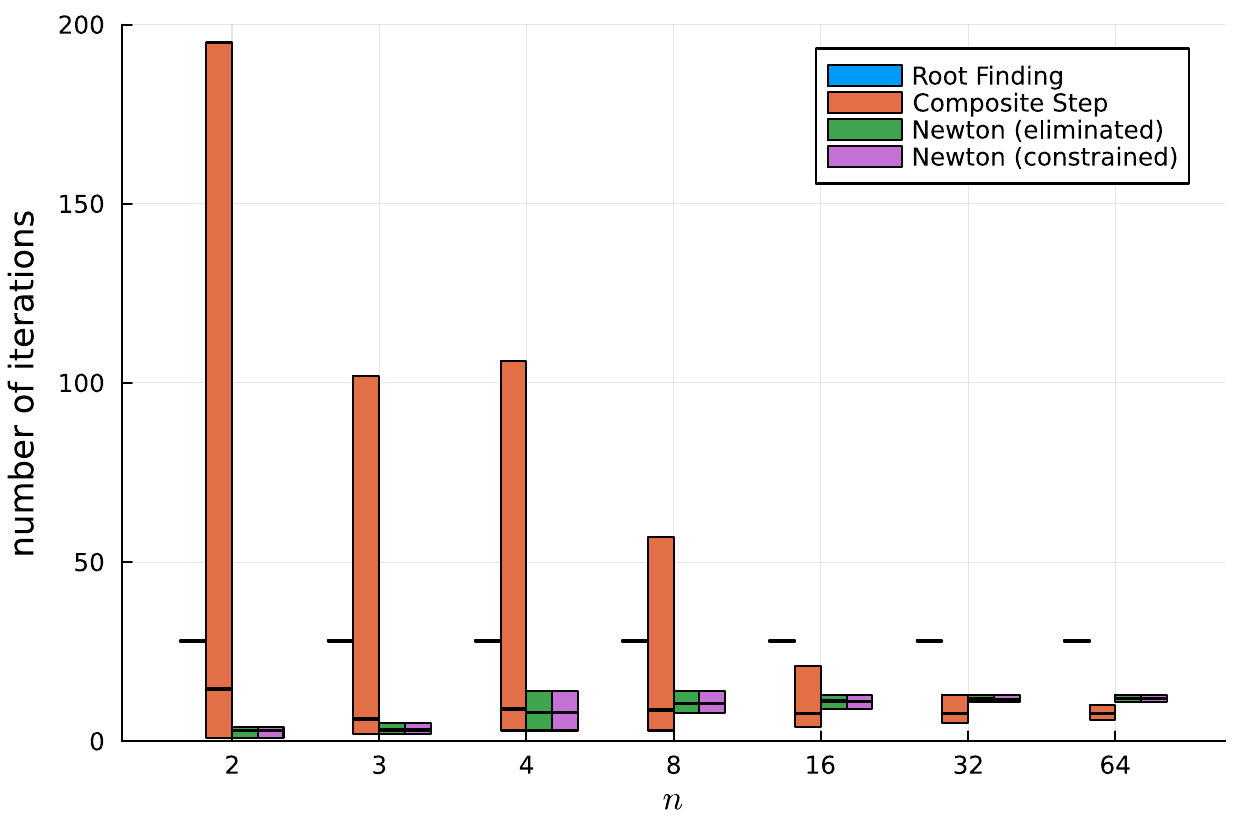}
 }
 \hspace{0.05\textwidth}
 \subfigure[$\mathcal{A}^{\overline{C}}$: Duplicate singular values]{
  \includegraphics[width=0.4\textwidth]{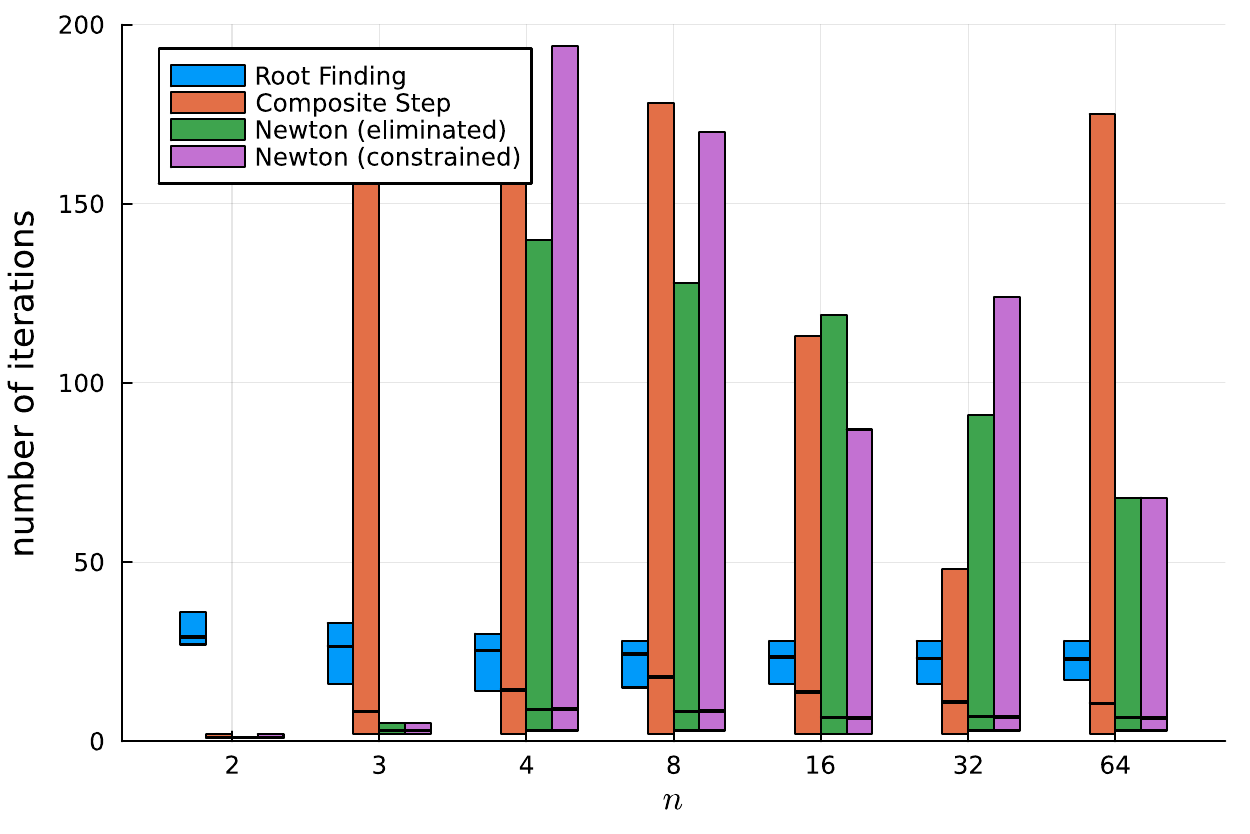}
 }
 \end{center}
 \caption{Iterations needed by the four algorithms on the four
 test data sets. Colored bars show the ranges between minimum and maximum
 number of iterations for each algorithm and matrix size.
 The small black horizontal lines show the average number of iterations.}
 \label{fig:iterations}
\end{figure}

In the tests we measured iteration counts and wall-times.
Figure~\ref{fig:iterations} shows the iterations needed for the algorithms to converge
to a solution. One can see that the behavior of the algorithms does depend on the
nature of the test matrices, but this dependency is different
for the four algorithms.

\begin{figure}
  \centering
  \includegraphics[width=0.8\textwidth]{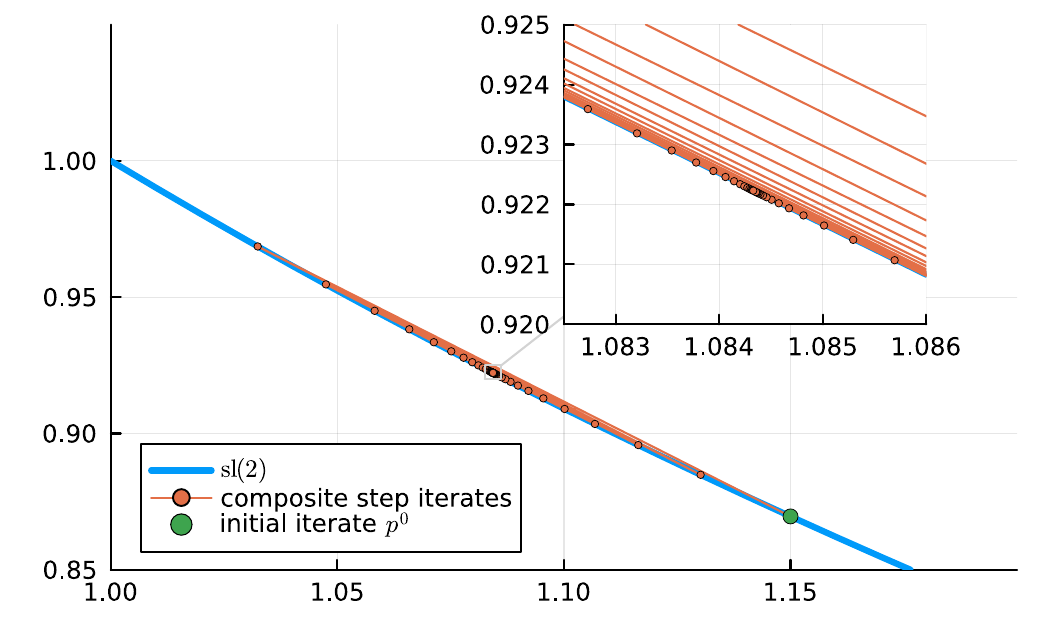}
  \caption{Visualization of a case of slow convergence of the composite step method
  for a point $a$ with $\product(a) <1$ (here: \(a=(0.3, 0)^T\)).
  The initial iterate constructed by Algorithm~\ref{alg:initial_value}
  is \(p^0=(1.15,1/1.15)^T\), and the first 50 iterates are displayed.}
  \label{fig:cs_slow_convergence}
\end{figure}

First of all,
the number of iterations of the root-finding algorithm is essentially constant
for all test sets and matrix sizes,
because the termination criterion implies an upper bound on the number of iterations
that does not depend on $A$.
The composite-step method always converges,
but for the convex case---the seemingly easier one,---it can require large
numbers of iterations. The large vertical bars
are caused by only a few outliers: The average number of iterations
(marked by the black horizontal line) is low.
Figure~\ref{fig:cs_slow_convergence} shows the behavior of the method
for such an outlier. The iterates alternate around the limit point,
but approach it only very slowly.
In the nonconvex case, the iteration numbers of the composite-step method
are much lower, and show no obvious dependence on the matrix size.

The two Netwon methods, on the other hand, perform much better
for matrices $A$ with $\det A < 1$ (including the case $\det A = 0$).
Iteration numbers are in the single-digit range or slightly above,
and do not seem to depend on the matrix size. For matrices with
a determinant above~1 they are still fast on average,
but here, also outliers can be observed where the Newton methods
require large iteration numbers.  In less than 1\,\% of all cases
the Newton methods reached the limit value of 200 iterations,
presumably because of indefinite or almost singular Hesse matrices.

The iteration numbers for the two Newton methods are similar. This is not
surprising: They produce identical iterates \cite[Chapter~10.2.3]{boyd_vandenberghe:2004},
and the different iteration numbers merely result from the fact that one method
uses a termination criterion in logarithmic coordinates, and other one
in hyperbolic coordinates.

\bigskip

\begin{figure}[h]
 \begin{center}
 \subfigure[$\mathcal{A}^{< 1}$: Convex projection]{
  \includegraphics[width=0.4\textwidth]{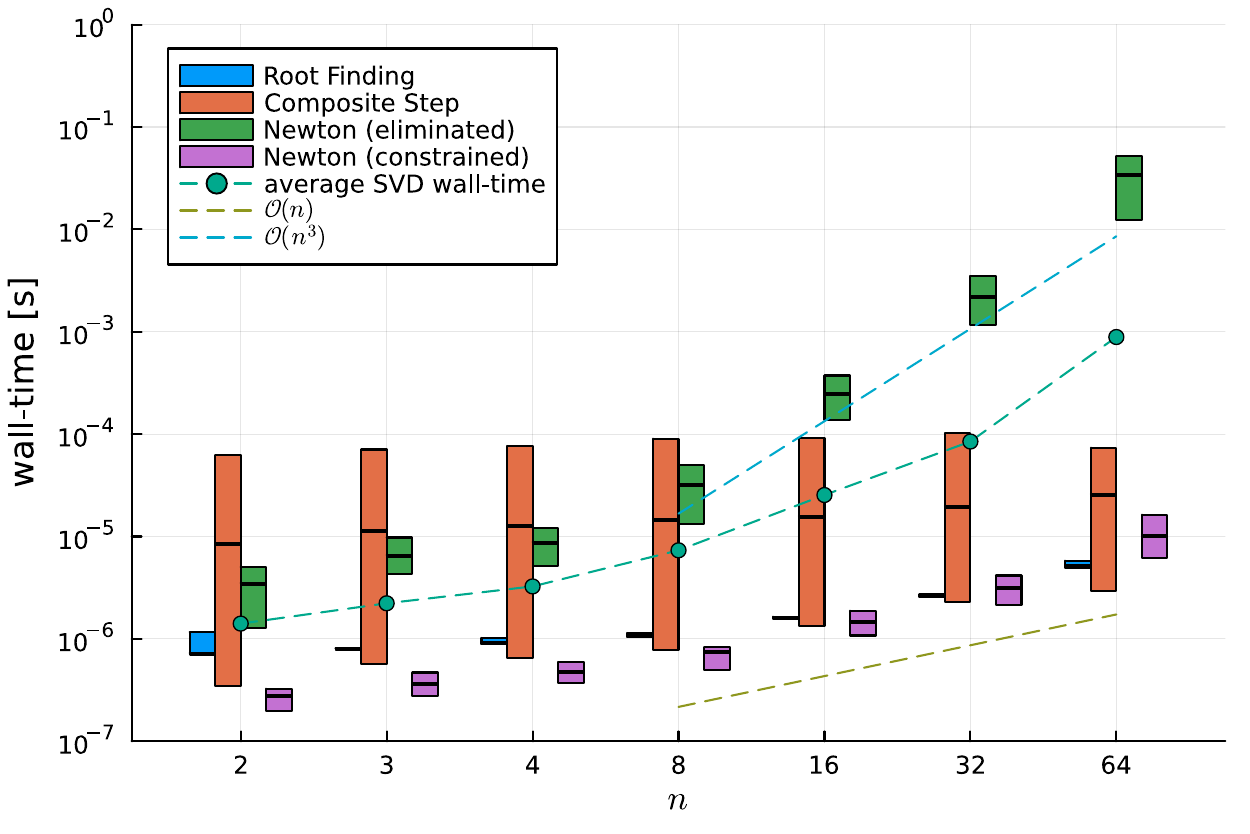}
 }
 \hspace{0.05\textwidth}
 \subfigure[$\mathcal{A}^{\ge 1}$: Non-convex projection]{
  \includegraphics[width=0.4\textwidth]{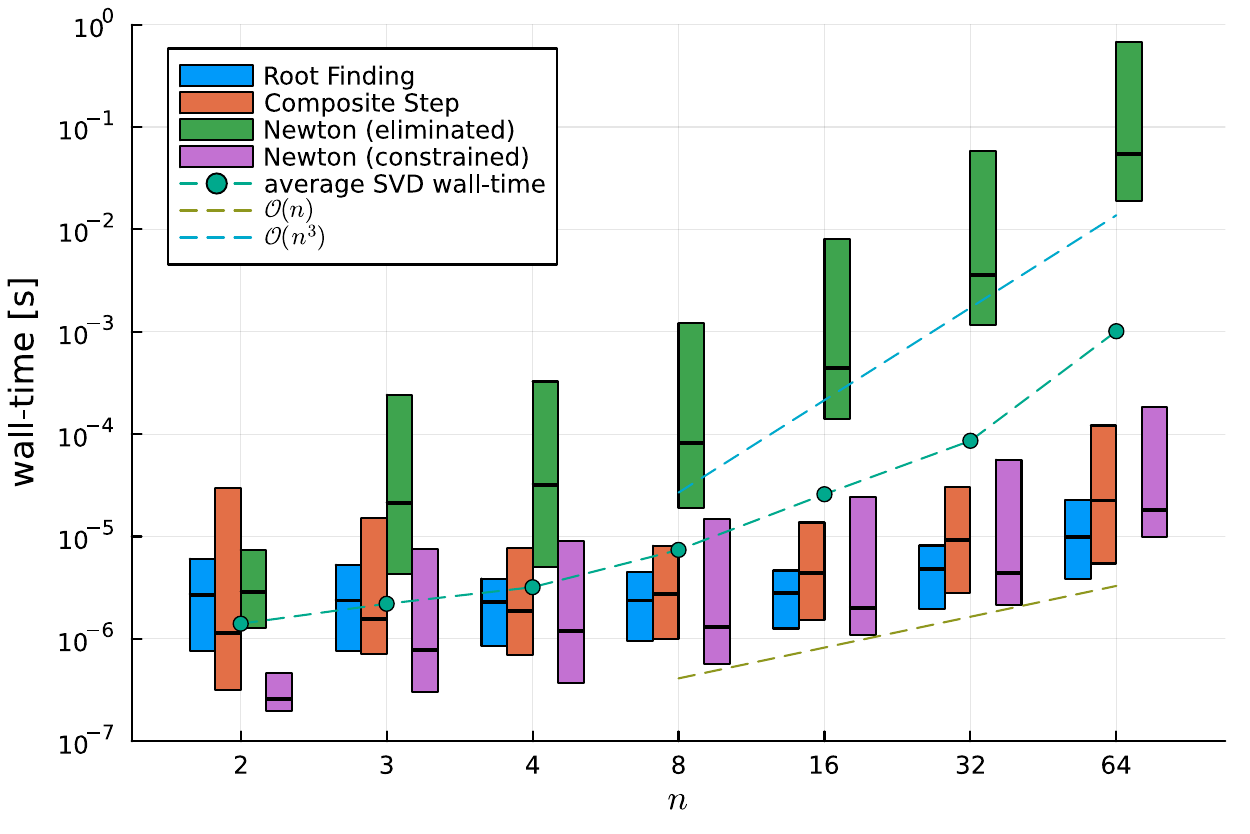}
 }

 \subfigure[$\mathcal{A}^0$: Singular matrices]{
  \includegraphics[width=0.4\textwidth]{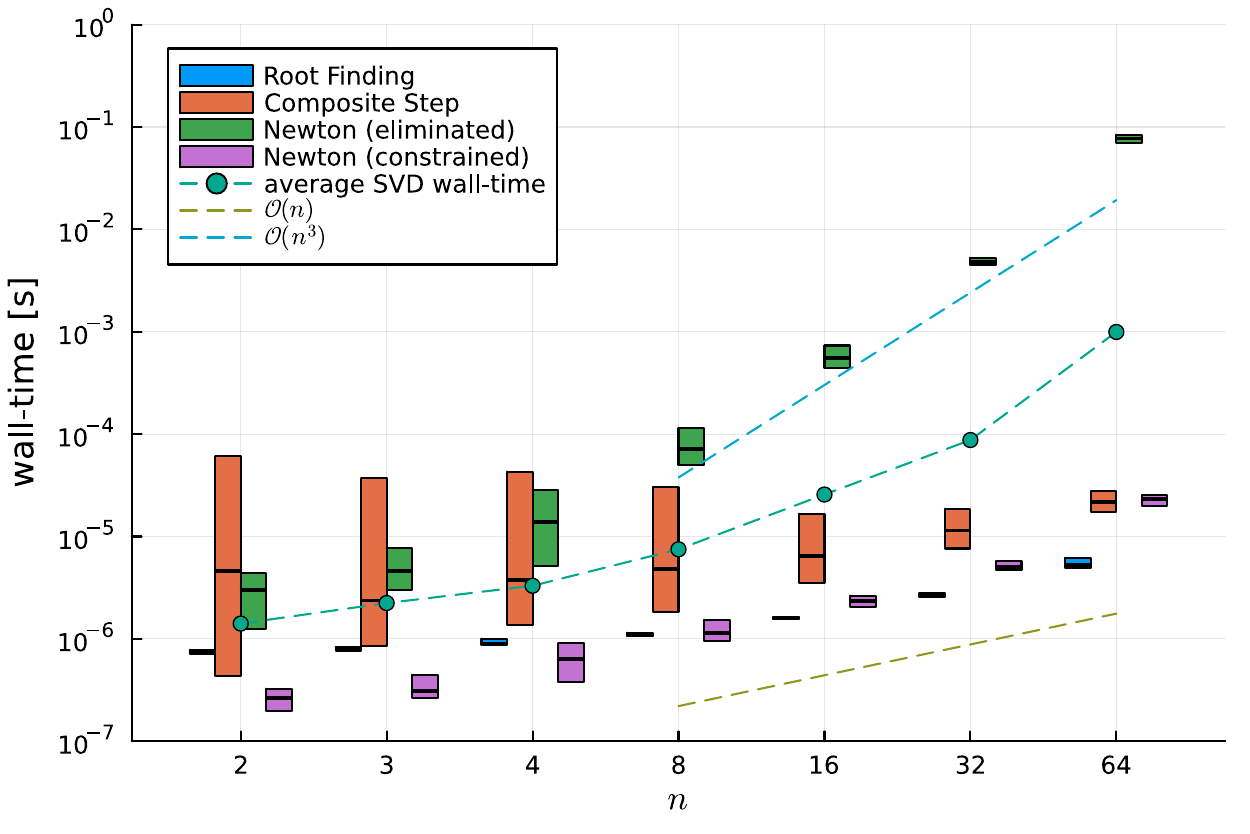}
 }
 \hspace{0.05\textwidth}
 \subfigure[$\mathcal{A}^{\overline{C}}$: Duplicate singular values]{
  \includegraphics[width=0.4\textwidth]{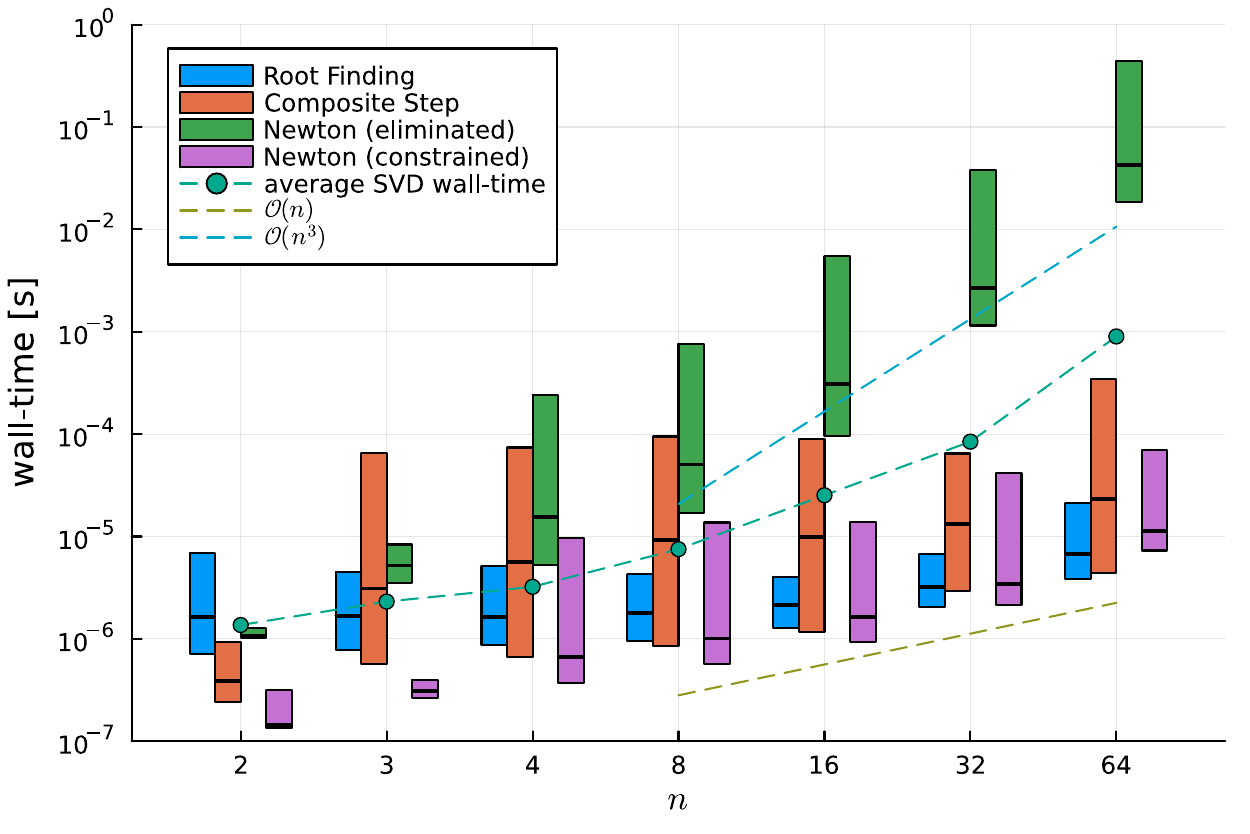}
 }
 \end{center}
 \caption{Wall-times needed by the four algorithms on the four
 test data sets. Colored bars show the ranges between minimum and maximum
 wall-times for each algorithm and matrix size.
 The small black horizontal lines show the average time.}
 \label{fig:wall_time}
\end{figure}

Figure~\ref{fig:wall_time} shows the wall times needed for the projections.
They were measured using \textsc{Julia}'s standard benchmarking package \texttt{BenchmarkTools.jl}.%
\footnote{\url{https://github.com/JuliaCI/BenchmarkTools.jl}}
That package runs each algorithm several times and determines the best-case run time.
By this, possible CPU speed variations are eliminated, and the effects of
\textsc{Julia}'s internal just-in-time precompilation are excluded.

As expected, the root-finding algorithm is fast, and the times
show very little variance. They appear to depend linearly on the matrix size,
which matches expectation.
The composite-step method turns out to be more competitive than
what could be expected from the iteration numbers, because each
iteration essentially only consists of a bisection and is therefore
still reasonably cheap.
The difference in behavior for the different test sets is now smaller.
The Newton methods, on the other hand, have to solve a linear system
of equations at each step, and the cost of this shows.
One can see that even though both methods need roughly the same number
of iterations, the unconstrained one is much slower.
This is because it has to factorize dense matrices, unlike
the constrained Newton method which can use the linear-time matrix inversion
formula of Lemma~\ref{lem:arrowhead_inversion}.
The cubic run-time behavior expected for the dense-matrix
factorization algorithm is clearly observed.

All four algorithms tested here involve a singular value decomposition to get the diagonal
representation~\eqref{eq:A-and-P-diagonalized}, and the triple matrix product $P = U \Sigma_P V^T$
to recover the result matrix from the minimizer in $\R^n$.
To illustrate the effect of this,
Figure~\ref{fig:wall_time} shows the time needed for the SVD and the matrix product separately.
The standard SVD implementation from \textsc{Julia}'s linear algebra package was used.
One can see that the SVD is cheaper than the Newton algorithm without constraints,
but (unless matrix sizes are very small) more expensive than all other algorithms.

\section{The derivative of the projection}
\label{sec:derivative}

In this final chapter we describe how to compute the first derivative of the projection
of a matrix $A \in \R^{n \times n}$ onto $\SL(n)$.  There is nothing particularly deep about this result,
but it is needed to implement $\SL(n)$-valued finite elements~\cite{grohs_hardering_sander_sprecher:2019}
and we include it for the readers' convenience.

\subsection{Computing the derivative}

While most of this manuscript has taken place in $\R^n_{\ge 0}$, the space of singular values,
now we go back to looking at actual matrices again.
From differential geometry we know that as $\SL(n)$ is a smooth embedded submanifold of $\R^{n \times n}$,
there exists a tubular neighborhood such that the closest-point projection
is differentiable~\cite[Prop.\,6.25 and Problem~6.5]{lee:2003}.
To compute the first derivative of the projection explicitly,
let $P \in \SL(n)$ be a stationary point of the squared distance functional
$\dist{A} : \SL(n) \to \R$ to the matrix
$A \in \R^{n \times n}$.  Then, $P$ and a Lagrange multiplier $\lambda \in \R$
solve the stationarity equation~\eqref{eq:lagrange}
of the Lagrange functional
\begin{equation*}
 0
 =
 \nabla \mathcal{L}(P,\lambda,A)
 =
 \begin{pmatrix}
  P - A + \lambda \nabla \det P \\
  \det P -1
 \end{pmatrix}
 \equalscolon
 \begin{pmatrix}
 F_1(P,\lambda,A) \\
 F_2(P,\lambda,A)
 \end{pmatrix}.
\end{equation*}

We interpret $P$ and $\lambda$ as functions of $A$.
To obtain the derivative of $P$ with respect to a small variation
$\delta A \in T_A \R^{n \times n} \cong \R^{n \times n}$ of $A$,
write $\partial_P$, $\partial_\lambda$, and $\partial_A$ for the
partial derivatives with respect to the first, second, and third
argument, respectively.  We compute the total derivatives of
$F_1$ and $F_2$ with respect to $\delta A$
and obtain the linear system
\begin{equation}
\label{eq:general_system_for_derivative}
 \begin{pmatrix}
  \partial_P F_1 & \partial_\lambda F_1 \\
  \partial_P F_2 & \partial_\lambda F_2
 \end{pmatrix}
 \begin{pmatrix}
  \delta P \\
  \delta \lambda
 \end{pmatrix}
 =
 -
 \begin{pmatrix}
  \partial_A F_1 \\
  \partial_A F_2
 \end{pmatrix},
\end{equation}
where $\delta P$ and $\delta \lambda$ denote the derivatives of $P$ and $\lambda$,
respectively, in the direction of $\delta A$.
The desired quantities $\delta P \in \R^{n \times n}$ and $\delta \lambda \in \R$
can therefore be obtained by solving a linear system,
provided the linear operator in~\eqref{eq:general_system_for_derivative} is invertible.

To fully understand the notation, note that the linear system is not an actual
matrix equation but rather an operator equation
with the linear operator given in $2 \times 2$ block form. For fixed $\delta A$,
the unknowns $\delta P$ and $\delta \lambda$ are in $\R^{n \times n}$ and~$\R$,
respectively.  The lower-left block $\partial_P F_2$ of~\eqref{eq:general_system_for_derivative} is a matrix,
which acts on the unknown $\delta P$ by Frobenius multiplication
$ \partial_P F_2 : \delta P
 \colonequals
 \sum_{i,j=1}^n \frac{\partial F_2}{\partial P_{ij}} (\delta P)_{ij}$.
The upper-right block $\partial_\lambda F_1$ is also a matrix,
and applying it to the scalar $\delta \lambda$ yields a matrix.
The quantity $\partial_\lambda F_2$ is a number, but
$\partial_P F_1$ is a fourth-order tensor, which we interpret as a
map $\R^{n \times n} \to \R^{n \times n}$.

Plugging in the specific forms of $F_1$ and $F_2$ of the projection problem we get
\begin{equation}
\label{eq:sensitivity_generic}
 \begin{pmatrix}
  (\cdot) + \lambda D(\nabla \det P)(\cdot) & (\det P) P^{-T} \\
  (\det P) P^{-T} : (\cdot) & 0
 \end{pmatrix}
 \begin{pmatrix}
  \delta P \\
  \delta \lambda
 \end{pmatrix}
 =
 \begin{pmatrix}
  \delta A \\
  0
 \end{pmatrix},
\end{equation}
where $D(\nabla \det P)(\cdot) : \R^{n \times n} \to \R^{n \times n}$
is the directional derivative of the first derivative of the determinant function.
The following lemma computes this directional derivative explicitly.

\begin{lemma}
\label{lem:second_derivative_of_det}
 The directional derivative of $\nabla \det P = (\det P) P^{-T}$ is the linear map
 \begin{align*}
  D(\nabla \det P)(\cdot) & \; : \; \R^{n \times n} \to \R^{n \times n},
  \\
  X & \mapsto \det P \operatorname{tr}(P^{-1} X) P^{-T} - \det P (P^{-T} X^T P^{-T}).
 \end{align*}
\end{lemma}
\begin{proof}
 Using the product rule we get for any $X \in \R^{n \times n}$
 \begin{align*}
  D(\nabla \det P)(X)
  & =
  D\big((\det P)P^{-T}\big)(X)
  =
  \big( D (\det P)(X) \big) P^{-T} + \det P D(P^{-T})(X).
 \end{align*}
 By Jacobi's formula, the first summand is $D(\det P)(X) = \det P \operatorname{tr}(P^{-1}X)$.
 For the second summand we use that
 \begin{equation*}
  0
  =
  \frac{d}{dt} \Big[ (P+tX)^T (P+tX)^{-T}\Big]\bigg|_{t=0}
  =
  \frac{d}{dt} (P^T+tX^T)\Big|_{t=0} P^{-T} + P^T \frac{d}{dt} (P+tX)^{-T}\Big|_{t=0},
 \end{equation*}
 which implies that
 \begin{equation*}
  D(P^{-T})(X) = - P^{-T} X^T P^{-T}.
 \end{equation*}
 Together we obtain the assertion.
\end{proof}

Plugging the expression for the derivative of $\nabla \det P$ into~\eqref{eq:sensitivity_generic}
and using $\det P = 1$ we obtain
\begin{equation}
\label{eq:sensitivity}
 \begin{pmatrix}
  (\cdot) + \lambda \big[ \operatorname{tr} (P^{-1} (\cdot)) P^{-T} - P^{-T}(\cdot)^T P^{-T} \big] & P^{-T} \\
  P^{-T} : (\cdot) & 0
 \end{pmatrix}
 \begin{pmatrix}
  \delta P \\
  \delta \lambda
 \end{pmatrix}
 =
 \begin{pmatrix}
  \delta A \\
  0
 \end{pmatrix}.
\end{equation}
Solutions $(\delta P, \delta \lambda)$ of this contain the derivative in the direction $\delta A$ of the projection of $A$ onto $\SL(n)$ as the first factor.

Before investigating the solvability of \eqref{eq:sensitivity} we first note that it can be
put into an equivalent form where $P$ is replaced by the diagonal matrix of its
singular values. Since the singular value decomposition is computed by all projection
algorithms of Chapter~\ref{sec:algorithms} anyway, the following diagonalized system
can be used to compute the derivative $\partial P$
more effectively.

\begin{lemma}
 Let $P = U \Sigma V^T$ be a singular value decomposition of $P$.
 Then \eqref{eq:sensitivity} has a solution $(\delta P, \delta \lambda) \in \R^{n \times n} \times \R$ if and only if
 \begin{equation}
 \label{eq:sensitivity_diagonal}
  \begin{pmatrix}
   (\cdot) + \lambda \big[ \operatorname{tr} (\Sigma^{-1} (\cdot)) \Sigma^{-1} - \Sigma^{-1}(\cdot)^T \Sigma^{-1} \big] & \Sigma^{-1} \\
   \Sigma^{-1} : (\cdot) & 0
  \end{pmatrix}
  \begin{pmatrix}
   \delta Y \\ \delta \eta
  \end{pmatrix}
  =
  \begin{pmatrix}
   U^T \delta A V \\
   0
  \end{pmatrix}
 \end{equation}
 has a solution $(\delta Y, \delta \eta) \in \R^{n \times n} \times \R$.
 The solutions are related by
 \begin{equation*}
  \delta Y = U^T (\delta P) V,
  \qquad
  \delta \eta = \delta \lambda.
 \end{equation*}

\end{lemma}
\begin{proof}
The singular value decomposition of $P$ implies
\begin{equation*}
 P^{-1} =  V \Sigma^{-1} U^T
 \qquad \text{and} \qquad
 P^{-T} = U \Sigma^{-1} V^T.
\end{equation*}
A direct computation then shows
\begin{align}
\label{eq:trace_diagonalization}
 P^{-T} : \delta P
 & =
 \operatorname{tr}(P^{-1} \delta P)
 =
 \operatorname{tr} (V \Sigma^{-1} U^T U \delta Y V^T) \\
 \nonumber
 & =
 \operatorname{tr} (V \Sigma^{-1}\delta Y V^T)
 =
 \operatorname{tr} (\Sigma^{-1}\delta Y)
 =
 \Sigma^{-1} : \delta Y.
\end{align}
In other words, $\delta P$ satisfies the second equation of~\eqref{eq:sensitivity}
if and only if $\delta Y$ satisfies the second equation of~\eqref{eq:sensitivity_diagonal}.

Plugging the identity~\eqref{eq:trace_diagonalization} into the first equation
of~\eqref{eq:sensitivity} yields
\begin{equation*}
 \delta P + \lambda \Big[ \operatorname{tr}(\Sigma^{-1} \delta Y) U \Sigma^{-1} V^T
 - (U \Sigma^{-1} V^T) (\delta P)^T (U \Sigma^{-1} V^T) \Big]
 + \delta \lambda \, U \Sigma^{-1} V^T
 = \delta A.
\end{equation*}
Multiplying by $U^T$ from the left and by $V$ from the right simplifies this to
\begin{equation*}
 U^T \delta P V
 + \lambda \Big[ \operatorname{tr}(\Sigma^{-1} \delta Y) \Sigma^{-1}
 - (\Sigma^{-1} V^T) (\delta P)^T (U \Sigma^{-1}) \Big]
 + \delta \lambda\, \Sigma^{-1}
 = U^T \delta A V.
\end{equation*}
With $\delta Y \colonequals U^T \delta P\, V$ and $\delta \eta = \delta \lambda$ we get
the first equation of~\eqref{eq:sensitivity_diagonal}.

Since all transformations applied to the first equation were multiplications
with invertible matrices,
the equation has a solution $(\delta Y, \delta \eta) \in \R^{n\times n} \times \R$
if and only if
the original equation has a solution $(\delta P, \delta \lambda) \in \R^{n \times n} \times \R$.
\end{proof}

\subsection{Well-posedness}

We now show that the linear system~\eqref{eq:sensitivity_diagonal}
(and hence~\eqref{eq:sensitivity}) is uniquely solvable, as long as $\lambda$ does not
assume certain ``forbidden'' values.  These values are related to symmetries
of the problem, and to second derivatives of the Lagrange functional.

To begin, note again that the second equation of \eqref{eq:sensitivity_diagonal} is
\begin{equation*}
 0 = \Sigma^{-1} : \delta Y.
\end{equation*}
Since $\Sigma^{-1}$ is a normal vector of $\SL(n)$ at the point $\Sigma$, we can
interpret this as a tangentiality condition.
Plugging it into the first equation yields
\begin{equation}
\label{eq:sensitivity_rotated}
 \delta Y - \lambda \big[ \Sigma^{-1} \delta Y^T \Sigma^{-1} \big]  + \Sigma^{-1} \delta \eta
 =
 U^T \delta A V.
\end{equation}

Define the operator
\begin{equation*}
 \skewsig(X)
 \colonequals
 X - \lambda \Sigma^{-1} X^T \Sigma^{-1},
\end{equation*}
i.e., the operator that makes up the first two terms of~\eqref{eq:sensitivity_rotated}.
Since the right-hand side $U^T \delta A\, V$ of~\eqref{eq:sensitivity_rotated} can be any matrix,
we have to show that
any matrix can be represented in the form $\skewsig(X) + \Sigma^{-1}\delta \eta$,
where $\delta \eta$ is a scalar and $X$ is a matrix that satisfies $\Sigma^{-1} : X = 0$.
For this, we first examine $\skewsig$ on the unrestricted set of all matrices $X \in \R^{n \times n}$.

\begin{lemma}
 If $\lambda \neq \pm \Sigma_i \Sigma_j$ and $\lambda \neq \Sigma_i^2$ for all $1 \le i \neq j \le n$,
 then $\skewsig : \R^{n \times n} \to \R^{n \times n}$ is a bijection.
\end{lemma}
\begin{proof}
As $\skewsig$ is a linear map between two vector spaces of the same dimension,
all we have to show is that $\skewsig(X) = 0$ implies $X = 0$.

For this, we look at the entries of $\skewsig(X)$ individually.
Let first $i \neq j$, and
recall that multiplication from the left by a diagonal matrix scales the matrix rows.
Then
 \begin{equation*}
  \skewsig(X)_{ij}
  =
  X_{ij} - \lambda(\Sigma^{-1} X \Sigma^{-1} )_{ji}
  =
  X_{ij} - \lambda \Sigma^{-1}_j X_{ji} \Sigma^{-1}_i
 \end{equation*}
 and
 \begin{equation*}
  \skewsig(X)_{ji}
  =
  X_{ji} - \lambda(\Sigma^{-1} X \Sigma^{-1} )_{ij}
  =
  X_{ji} - \lambda \Sigma^{-1}_i X_{ij} \Sigma^{-1}_j.
 \end{equation*}
 Both entries being zero implies $X_{ij} = X_{ji} = 0$ only if the linear system
 \begin{equation*}
  \begin{pmatrix}
   1 & -\lambda \Sigma_i^{-1} \Sigma_j^{-1} \\
   -\lambda \Sigma_i^{-1} \Sigma_j^{-1} & 1
  \end{pmatrix}
  \begin{pmatrix} X_{ij} \\ X_{ji} \end{pmatrix}
  =
  \begin{pmatrix} 0 \\ 0 \end{pmatrix}
 \end{equation*}
 has a unique solution.  This is the case if
 \begin{equation*}
  \det
    \begin{pmatrix}
   1 & -\lambda \Sigma_i^{-1} \Sigma_j^{-1} \\
   -\lambda \Sigma_i^{-1} \Sigma_j^{-1} & 1
  \end{pmatrix}
  =
  1 - \lambda^2 \Sigma_i^{-2} \Sigma_j^{-2}
  \neq
  0,
 \end{equation*}
 which is equivalent to $\lambda \neq \pm \Sigma_i \Sigma_j$.

 For the diagonal entries we get
 \begin{equation*}
  \skewsig(X)_{ii}
  =
  X_{ii} - \lambda \Sigma_i^{-2} X_{ii}
  =
  (1 - \lambda \Sigma_i^{-2})X_{ii}.
 \end{equation*}
 Hence $\skewsig(X)_{ii} = 0$ implies $X_{ii} = 0$ if $\lambda \neq \Sigma_i^2$.
\end{proof}

\begin{remark}
 The restrictions on the Lagrange multiplier~$\lambda$ all have interpretations.
 For example,
 the value $\lambda = \Sigma_i \Sigma_j$, $i \neq j$ corresponds to the case where
 $A$ has duplicate singular values.  Indeed, if $A = U \Sigma_A V^T$ is the singular
 value decomposition of $A$, the stationarity condition
 \eqref{eq:lagrange_diagonal} implies
 \begin{equation*}
  \Sigma_A = \Sigma + \lambda \Sigma^{-1},
 \end{equation*}
 and if $\lambda = \Sigma_i \Sigma_j$ for some $i \neq j$ we get
 \begin{align*}
  \Sigma_{A,i}
  =
  \Sigma_i + \lambda \Sigma_i^{-1}
  =
  \Sigma_i + \Sigma_i \Sigma_j \Sigma_i^{-1}
  =
  \Sigma_i + \Sigma_j
  =
  \lambda \Sigma_j^{-1} + \Sigma_j
  =
  \Sigma_{A,j}.
 \end{align*}
 Similarly, $\lambda = - \Sigma_i \Sigma_j$ leads to
 \begin{align*}
  \Sigma_{A,i}
  =
  -\Sigma_{A,j},
 \end{align*}
 which can only happen if $A$ is not invertible, because the singular values
 of a matrix are never negative.

 If $\lambda = \Sigma_i^2$ for some $i=1,\dots,n$ then the bordered Hesse matrix
 of the constrained Newton system~\eqref{eq:constrained_newton_system}
 is not invertible at $(\Sigma,\lambda)$, because the $i$-th
 entry on the diagonal of its upper left block is
\begin{equation}
\label{eq:connection_to_hessian}
 \exp \xi_i (2 \exp \xi_i - a_i)
 =
 2 \Sigma_i^2 - \Sigma_i \Sigma_{A,i}
 =
 \Sigma_i^2 - (\Sigma_i \Sigma_{A,i} - \Sigma_i^2)
 =
 \Sigma_i^2 - \lambda
 =
 0.
\end{equation}
\end{remark}

As a corollary we obtain directly that~\eqref{eq:sensitivity_rotated} may not
have a solution if $\lambda = \pm \Sigma_i \Sigma_j$ for some $i \neq j$.
Indeed, in that case $\skewsig(X)_{ij} = \mp \skewsig(X)_{ji}$, and therefore
also $(\skewsig(X) + \Sigma^{-1} \partial\eta)_{ij} = \mp (\skewsig(X) + \Sigma^{-1}\partial \eta)_{ji}$,
because $\Sigma^{-1} \delta \eta$ is diagonal.
This excludes a solution if $(U^T\delta A V)_{ij} \neq \mp (U^T\delta A V)_{ji}$.

\bigskip

We have shown that $\skewsig$ is surjective as a map from $\R^{n \times n}$
onto itself.
 However, as solutions of~\eqref{eq:sensitivity_diagonal} we can only
consider those matrices $X$ that also fulfill the second equation,
i.e., the ones with $\Sigma^{-1} : X = 0$.  As $\skewsig$ is a bijection,
the image of the $n^2-1$-dimensional space of all $X \in \R^{n \times n}$ with $\Sigma^{-1} : X = 0$
is an $n^2 -1$-dimensional space.  We compute the orthogonal complement
of this image space. In the following we write $\Sigma^{-1} \perp X$ as a shorthand
for $\Sigma^{-1} : X = 0$.

\begin{lemma}
\label{lem:range_of_skew}
 Let $X \perp \Sigma^{-1}$ and $\lambda \neq \Sigma_i^2$ for all $i=1,\dots,n$.
 Then $(\Sigma -\lambda \Sigma^{-1})^{-1}$ is a well-defined matrix, and
 \begin{equation*}
  \skewsig(X) \perp (\Sigma - \lambda \Sigma^{-1})^{-1}.
 \end{equation*}
\end{lemma}

\begin{proof}
 Note that the matrix $(\Sigma - \lambda \Sigma^{-1})^{-1}$ is diagonal.
 Therefore, for any $X \in \R^{n \times n}$,
 \begin{align*}
  \skewsig(X) : (\Sigma - \lambda \Sigma^{-1})^{-1}
  & =
  \sum_{i=1}^n (1 - \lambda \Sigma^{-2}_i)X_{ii} (\Sigma_i - \lambda\Sigma^{-1}_i)^{-1}  \\
  & =
  \sum_{i=1}^n (1 - \lambda \Sigma^{-2}_i) X_{ii} \frac{\Sigma^{-1}_i}{1 - \lambda \Sigma^{-2}_i} \\
  & =
  \sum_{i=1} X_{ii} \Sigma^{-1}_i \\
  & =
  X : \Sigma^{-1} \\
  & =
  0.\qedhere
 \end{align*}
\end{proof}

Using this orthogonality relation we show that $\Sigma^{-1}$ is not contained
in the image of $\skewsig$.  This implies that $\Sigma^{-1}$ and the range
of $\skewsig$ span $\R^{n \times n}$ even if only matrices $X$ orthogonal
to $\Sigma^{-1}$ are considered as arguments of $\skewsig$.

\begin{lemma}
 \label{lem:range_does_not_contain_normal}
 Assume that $\Sigma$ and $\lambda$ are such that $\sum_{i=1}^n \frac{1}{\Sigma_i^2 - \lambda}$
 is well-defined and non-zero.  Then $X \perp \Sigma^{-1}$ implies $\skewsig(X) \neq \Sigma^{-1}$.
\end{lemma}

\begin{proof}
 We show that $\Sigma^{-1}$ is not orthogonal to the orthogonal complement
 of the image of $\skewsig$.  By Lemma~\ref{lem:range_of_skew}, that orthogonal complement
 is spanned by $(\Sigma - \lambda \Sigma^{-1})^{-1}$. Testing for orthogonality
 of that with $\Sigma^{-1}$, we obtain
 \begin{align*}
  \Sigma^{-1} : (\Sigma - \lambda \Sigma^{-1})^{-1}
  =
  \sum_{i=1}^n \frac{1}{\Sigma_i^2 - \lambda},
 \end{align*}
 which is non-zero under the given conditions.
\end{proof}

\begin{remark}
The condition $\sum_{i=1}^n \frac{1}{\Sigma_i^2 - \lambda} \neq 0$
corresponds to the invertibility condition~\eqref{eq:trace_condition_hyperbolic_2}
for the bordered Hesse matrix of the constrained Newton system
of Chapter~\ref{sec:newton_with_constraints}. Indeed,
reusing the computation~\eqref{eq:connection_to_hessian} we get
\begin{align*}
 \sum_{i=1}^n \frac{1}{\exp \xi_i (2 \exp \xi_i - a_i)}
 =
 \sum_{i=1}^n \frac{1}{\Sigma_i^2 - \lambda}.
\end{align*}
\end{remark}

Summing up, we have shown:

\begin{theorem}
 The linear system~\eqref{eq:sensitivity_diagonal} is well posed, provided that
 $\lambda \neq \Sigma_i^2$, $1 \le i \le n$, $\lambda \neq \pm \Sigma_i \Sigma_j$,
 $1 \le i \neq j \le n$, and $\sum_{i=1}^n (\Sigma_i^2-\lambda)^{-1} \neq 0$.
\end{theorem}

%%%%%%%%%%%%%%%%%%%%%%%%%%%%%%%%%%%%%%%%%%%%%%%%%%%%%%%%%%%%%%%%%%%%%%%%
\printbibliography

\end{document}